\documentclass{amsart}
\usepackage[utf8]{inputenc}

\title{New Approaches to the Monotonicity Inequality for Linear Stochastic PDEs}

\author{Suprio Bhar}
\address{Suprio Bhar, Department of Mathematics and Statistics, Indian Institute of Technology Kanpur, Kalyanpur, Kanpur - 208016, India.}
\email{suprio@iitk.ac.in}

\author{Arvind Kumar Nath}
\address{Arvind Kumar Nath, Department of Mathematics and Statistics, Indian Institute of Technology Kanpur, Kalyanpur, Kanpur - 208016, India.}
\email{yarvind@iitk.ac.in}

\date{}
\usepackage{setspace}
\usepackage{geometry}
\usepackage{enumerate}
\usepackage{enumitem, xcolor, amssymb,latexsym,amsmath,bbm}
\usepackage{mathtools}
\usepackage[mathscr]{euscript}
\usepackage{amsthm}
\usepackage{amssymb}
\usepackage{enumitem}
\usepackage{cite}
\newcommand{\R}{\mathbb{R}}
\newcommand{\C}{\mathbb{C}}

\newcommand{\Z}{\mathbb{Z}}
\newcommand{\Sc}{\mathcal{S}}
\newcommand{\N}{\mathbb{N}}

\newcommand{\Hop}{\mathbf{H}}
\newcommand{\D}{\mathcal{D}}
\newcommand{\inpr}[3][]{\left\langle#2 \,,\, #3\right\rangle_{#1}}

\newcommand{\norm}[2][]{{\!\|#2\|\!_{#1}}}

\newtheorem{theorem}{Theorem}[section]

\newtheorem{lemma}[theorem]{Lemma}
\newtheorem{proposition}[theorem]{Proposition}
\theoremstyle{definition}
\newtheorem{definition}[theorem]{Definition}
\theoremstyle{remark}
\newtheorem{remark}[theorem]{Remark}

\numberwithin{equation}{section}

\allowdisplaybreaks[4]

\begin{document}
\begin{abstract}
The Monotonicity inequality is an important tool in the understanding of existence and uniqueness of strong solutions for Stochastic PDEs. In this article, we discuss three approaches to establish this deterministic inequality explicitly.
\end{abstract}


\maketitle

\section{introduction and Main results}\label{s1:intro}
The Monotonicity inequality (see \cite{MR570795,MR2560625,MR1465436}), similar to the Coercivity inequality, is an important tool in the analysis of stochastic PDEs (see \cite{MR1207136,MR2329435,MR2295103,MR2560625, MR688144,MR4600818,MR4583740,MR2674056,MR2857016,MR2765423,MR2111320,MR1985790,MR1500166,MR1242575,MR771478,MR1465436,MR2370567,MR1661766,MR570795,MR705933,MR553909,MR2356959,MR1135324,MR876085,MR3235846,MR3236753}), primarily focusing on the uniqueness of strong solutions. Its application in obtaining the existence and uniqueness of SPDEs was further explored in \cite{MR2479730, brajeev-arxiv, MR3063763, MR2560625, MR2590157, MR3331916, MR4568882, MR1465436}. In particular, \cite{MR2590157} utilized this inequality to demonstrate the existence and uniqueness of strong solutions to linear stochastic PDEs in the dual of a countably Hilbertian Nuclear space. 

We now focus on the Schwartz space $\Sc(\R)$ (see \cite{MR1681462, MR2296978, MR771478, MR1157815}) of real-valued rapidly decreasing smooth functions on $\R$, which has a countably Hilbertian Nuclear topology (see \cite{MR771478}). We recall the details in Section \ref{S2:description} and work with stochastic PDEs in the space $\Sc^\prime(\R)$ of tempered distributions, which is dual to $\Sc(\R)$. The countably Hilbertian topology is described in terms of the Hermite-Sobolev norms $\|\cdot\|_p$, associated to the separable Hilbert spaces, the so-called Hermite-Sobolev spaces $\Sc_p(\R), p \in \R$ (see Section \ref{S2:description}). In this article, we attempt to explicitly establish the Monotonicity inequality for a pair of `adjoint' operators.

Let $\sigma, \;  b $ be real valued smooth functions with bounded
derivatives. We define operators $ A ^ * , L ^ * $ on $\Sc^\prime(\R)$ as follows:
     \begin{align*}
          A ^ \ast \phi = & - \partial ( \sigma \phi ) &   L ^ \ast \phi = &
\frac { 1 }
 { 2 }  \partial ^ 2
( \sigma ^ 2 \phi ) - \partial ( b \phi ), \forall \phi \in \Sc^\prime(\R).
     \end{align*}
Note that $A \phi \in \Sc(\R)$ and $L\phi \in \Sc(\R)$, whenever $\phi \in \Sc(\R)$. Fix $p \in \R$, we say that the  pair of linear operators $(L^\ast, A^\ast)$ satisfies
the Monotonicity inequality in $\|\cdot\|_p$, if we have
\begin{equation}\label{Monotoniticity-inequality-chap7}
2\inpr[p]{\phi}{L^\ast\phi} + \|A^\ast\phi\|_{p}^2\leq C\|\phi\|^2_p,\quad \forall
\phi\in \Sc(\R),
\end{equation}
where $C=C(\sigma,b,p) > 0$ is a positive constant.

The Monotonicity inequality for $(L^\ast, A^\ast)$ has been proved explicitly in the following cases.

    \begin{theorem}\label{old-results}
    For every $p\in \R$, the Monotonicity inequality \eqref{Monotoniticity-inequality-chap7} for the pair $(L^\ast,A^\ast)$ holds for the following cases:
    \begin{enumerate}[label=(\roman*)]
        \item \cite[Theorem 2.1]{MR2590157}. $\sigma \in \R$ and $b \in \R$.
        \item\cite[Theorem 4.6]{MR3331916}. $\sigma\in \R$ and $b(x):=b_0+Mx$ on $\R$ for fixed $b_0\in \R, M\in\R$.
    \end{enumerate}
\end{theorem}
Some results involving non-constant multipliers are also available in the literature, see for example, \cite[Theorem 4.4]{MR4117986} and \cite{alt-mono}.

In this article, we discuss three further approaches to the Monotonicity
inequality in the usual Hermite-Sobolev norm $\|\cdot\|_p$.  A brief overview of the proposed approaches are described in the subsections below. We discuss
these in detail in the later sections. However, our motivation is to highlight these approaches and as such, to avoid notational complexity, we shall state and prove all the results in the one-dimensional setting only. The results extend in a natural manner to their higher-dimensional analogues.

\subsection{Approach using Bounded Operators} 
 Recall that the derivative operator $\partial$ is a
densely defined, closed, linear, unbounded operators on the Hermite-Sobolev
spaces $\Sc_p(\R)$ (see \cite{MR1215939, MR3331916, MR771478}). Motivated by the integration by parts formula, a special structure of the
adjoint of these operators were identified in \cite[Theorem 2.1]{MR3331916}, namely that the adjoint is the negative of the derivative operator and perturbed by a specific bounded operator. Consequently, the Monotonicity inequality was proved when $\sigma$ is a constant and $b$ is in an affine form (see \cite[Theorem 4.2]{MR3331916}) for a class of adjoint operators. We revisit this result for the identification (see Lemma \ref{adj-partial-on-complex}) of the adjoint of the derivative operator and extend this identification to that on the Complex-valued Hermite-Sobolev
spaces $\Sc_p(\R; \C)$ and further extend to Lemma \ref{bndop1}. A key observation in this result is the usage of a `conjugate' with the Hermite operator $\Hop$ (defined in the Subsection \ref{S:Hermite-Op} below). Consequently, we prove the Monotonicity inequality for the case when both $\sigma$ and $b$ can be taken to be in the affine form (see Theorem \ref{Mono-ineq-affine-Form}), extending \cite[Theorem 4.2]{MR3331916}. 
\begin{theorem}\label{Mono-ineq-affine-Form}
Fix $\alpha, \beta, \gamma, \delta \in \R$ and set $\sigma  ( x )= \alpha x + \beta, \forall x \in \R$  and $b ( x ) = \gamma x + \delta, \forall x \in \R$.
Then, for all $ p \in\R$, the Monotonicity inequality \eqref{Monotoniticity-inequality-chap7} holds for the pair $(L^\ast, A^\ast)$.
 \end{theorem}
We have  used bounded operators to get this result which are mentioned in Section \ref{S2:proofs} with proofs.

\subsection{Approach using Good terms}
 In \cite{MR3331916}, the recurrence relations on the derivative of Hermite functions $h_n$   and the multiplication by co-ordinate functions   were crucially used to obtain the Monotonicity inequality \cite[Theorem 4.2]{MR3331916}. However, as we do not have any recurrence relation for multiplication by general smooth multipliers, this leads to a difficulty in extending the technique of \cite{MR3331916} to handle smooth multipliers in the usual norm $\norm[p]{\cdot}$.  In this approach, we attempt to obtain the Monotonicity inequality in the usual norm $\norm[p]{\cdot}$ by identifying suitable `good' terms.
 
 \begin{definition}\label{good-term-chap7}
Let \( \sigma_1 \) and \( \sigma_2 \) be real-valued smooth functions with
bounded derivatives. Let $p \in \Z_+$. We refer to the term
\[
\inpr[0]{\partial^{k_1} \sigma_1 \cdot x^{\alpha_1} \partial^{\beta_1}
\phi}{\partial^{k_2} \sigma_2 \cdot x^{\alpha_2} \partial^{\beta_2} \psi} \quad
\forall \phi, \psi \in \mathcal{S}(\R)
\]
as a good term with respect to $ \inpr[p]{\cdot}{\cdot}$, if
$\alpha_1, \alpha_2, \beta_1, \beta_1, k_1$ and $k_2$ are non-negative integers satisfying
\[
\alpha_1 + \beta_1 + 1_{\{0\}}(k_1) + 1_{\{0\}}(k_2) + \alpha_2 + \beta_2 \leq
4p.
\]
A finite sum of good terms shall be abbreviated as `FSGT'. 

We shall refer to the terms of the form $\partial ^ { k} \sigma  . x
^ { \alpha } \partial ^ { \beta  } \phi$ with $\alpha + \beta + 1_{\{0\}}(k) \leq p$ as `Terms of order $\leq p$'.
\end{definition}

 \begin{lemma}\label{good-term-bnd-chap7}
  A bilinear form given by a good term is bounded, i.e. if $ \sigma _ 1 , \;  \sigma
_ 2 $ are real valued smooth functions with bounded derivatives, then the bilinear
form $\inpr [0] { \cdot } { \cdot } : \Sc(\R) \times \Sc(\R) \to \R $ defined by
  \[ ( \phi , \psi ) \mapsto  \inpr [ 0 ] { \partial ^ { k _ 1 } \sigma _ 1 . x
^ { \alpha _ 1 } \partial ^ { \beta _ 1 } \phi } { \partial ^ { k _ 2 } \sigma _
2 . x ^ { \alpha _ 2} \partial ^ { \beta _ 2 } \psi } \; \forall \phi , \psi \in
\Sc(\R) \]
  is bounded w.r.t. $\norm[p]{ \cdot}$,  if $ \alpha _ 1 + \beta _ 1 + 1 _ { \{
0 \} } ( k _ 1 ) + 1 _ { \{ 0 \} } ( k _ 2 ) + \alpha _ 2 + \beta _ 2 \leq  4 p
$.
 \end{lemma}
 
 This approach relies on the fact that $\norm[p]{\phi}=\norm[0]{\Hop^p\phi}, \forall \phi \in \Sc_p(\R)$ (see Proposition \ref{Hop-properties}) and we then use a recurrence relation for $\Hop^p\phi$ given by the following result. 

 \begin{lemma}\label{H^P-lemma}
Let $p \geq 3$ be an integer. For any $\phi \in \Sc(\R)$, we have
    \begin{align*}
         \Hop ^ p  \phi =& \sum _ {  j    =   0   } ^ { p } ( - 1 ) ^ { p - j }
\binom { p } { j } x ^ {  2 j  } \partial ^ { 2  (  p - j ) } \phi \\
&- 2 \sum _ {
r = 1  } ^ { p - 1 } \left ( \sum _ { j = 1} ^ { r }   ( - 1 ) ^{ r - j }
\binom { r } { j }  . 2 j . x ^ { 2 j - 1 } \partial ^ { 2  ( r - j ) + 1  }  \Hop
^  {  p -  1 - r } \phi  \right ) + ( \text{Terms of order } \leq  2 p - 4).
    \end{align*}
    Here, terms of the form $x
^ { \alpha } \partial ^ { \beta  } \phi$ with $\alpha + \beta \leq 2p - 4$ have been referred to as `Terms of order $\leq 2p-4$'.
\end{lemma}

 We establish the Monotonicity inequality  
 involving smooth multipliers, under relevant sufficient conditions (see Theorem \ref{Mono-ineq-gen-functin}), with non-negative integer $p$ regularities. As this approach requires an iteration scheme, we first establish the result in $\norm[p]{\cdot}$ for the base cases $p = 0, 1, 2$ explicitly and then run the iteration.
\begin{theorem}\label{Mono-ineq-gen-functin}
    Let $\sigma,\; b$ be real valued smooth functions with bounded derivatives. Then, for $ p \in \Z_+$, the Monotonicity inequality \eqref{Monotoniticity-inequality-chap7} holds for the pair $(L^\ast, A^\ast)$ in $\|\cdot\|_p$.
 \end{theorem}

 The results obtained here closely parallels results in \cite{alt-mono}.

\subsection{Approach using an interpolation argument via the Three Lines Lemma}
 One of the shortcoming of the Approach II, discussed above, is the need to prove the result in some base case. However, a direct proof for fraction values of $p$, say $p = \frac{1}{2}$, seems  difficult to obtain, based on the current state-of-the-art. As such, it does not seem feasible to directly extend Approach II for non-integer $p$ regularities. We develop a new method (see Theorem \ref{proof-by-3-line lemma in affine case}) using interpolation via the well-known Three Lines Lemma (see Theorem \ref{The three lines lemma}) by which the inequality can be extended from non-negative integer values of $p$ to positive fractional values of $p$.
 
 \begin{theorem}[Three Lines Lemma \cite{MR304972}]\label{The three lines lemma}
Suppose $ F $ is a bounded continuous complex-valued function on the closed strip $
S = \left \{x + i y  = z \in \C : \; 0 \leq x \leq 1 \right \} $ that is
analytic in the interior of $ S $. If $ \left | F (  i y ) \right | \leq m _ 0 $
and $ \left | F ( 1 + i y ) \right | \leq m _ 1 $ for all $ y \in \R $,  then $
\left | F ( x + i y ) \right | \leq m _ 0 ^ { 1 - x } m  _ 1 ^ x  $ for all $ z
=  x + i y  \in S $.
\end{theorem}

\begin{theorem}\label{proof-by-3-line lemma in affine case}
  Fix $\alpha, \beta, \gamma, \delta \in \R$ and set $\sigma  ( x )= \alpha x + \beta, \forall x \in \R$  and $b ( x ) = \gamma x + \delta, \forall x \in \R$.
Then, for all non-negative non-integer values of $p$, the Monotonicity inequality \eqref{Monotoniticity-inequality-chap7} holds for the pair $(L^\ast, A^\ast)$ in $\|\cdot\|_p$. 
  \end{theorem}

 At the present stage of this approach, we are able to handle affine multipliers $\sigma$ and $b$; but we hope to extend the methodology to incorporate more general multipliers  in a future work. Theorem \ref{proof-by-3-line lemma in affine case} yields an alternative proof of Theorem \ref{Mono-ineq-affine-Form}, for non-negative non-integer $p$ regularities. This approach has been inspired by the usage of Three Lines Lemma in the proof of \cite[Theorem 2.1]{MR1999259}.

A brief outline of the article is as follows. In Section \ref{S2:description}, we explore the countably Hilbertian topology for Schwartz class functions, focusing on the real-valued case in Subsection \ref{S2:topology-Real} and the complex-valued case in Subsection \ref{S2:topology-complex}. We also recall standard results on the Hermite operator in Subsection \ref{S:Hermite-Op}. Section \ref{S2:proofs} is devoted to the proofs. Subsection \ref{Preliminary Results} contains some preliminary results and estimates. Proofs related to the three approaches I, II and III are described in Subsections \ref{First approach}, \ref{second approach} and \ref{third approach} respectively.

\section{Topology on the Schwartz space}\label{S2:description}

\subsection{ Countably Hilbertian Topology on the Real valued Schwartz class functions}\label{S2:topology-Real}
Let $\Sc(\R)$ be the space of smooth rapidly decreasing real-valued functions on $ \R $ with the topology given by L. Schwartz (see \cite{MR1681462, MR2296978, MR771478, MR1157815} and the references therein). This space, also known as the Schwartz space, is defined as follows:

\[
\Sc(\R) = \left\{ \varphi \in C^\infty(\R) : \forall k \geq 1, \max_{|\alpha| \leq k} \sup_{x \in \R} (1 + |x|^2)^k |\partial^\alpha \varphi(x)| < \infty \right\}.
\]

 $\Sc(\R)$ can be represented as the intersection of the Hermite-Sobolev spaces $\Sc_p(\R)$ for $p \in \R$, that is, $\Sc(\R) = \bigcap_{p \in \R} \Sc_p(\R)$. The Hermite-Sobolev spaces $\Sc_p(\R)$, for $p \in \R$, are defined as the completions of the Schwartz space $\Sc(\R)$ with respect to the norm $\| \cdot \|_p$ related to the following  pre-inner products:
\[
\langle f, g \rangle_p = \sum_{k = 0}^\infty  (2 k + 1)^{2p} \langle f, h_k \rangle_0 \langle g, h_k \rangle_0, \quad f, g \in \Sc(\R).
\]
Here, $\langle \cdot, \cdot \rangle_0$ represents the standard inner product in $\mathcal{L}^2(\R, dx)$. In this framework, $\{ h_k \}_{k = 0}^\infty$ denotes an orthonormal basis for $\mathcal{L}^2(\R, dx)$, consisting of Hermite functions.  The Hermite functions are given by:

\[
h_k(t) = (2^k k! \sqrt{\pi})^{-\frac{1}{2}} \exp{\left(-\frac{t^2}{2}\right)} H_k(t),
\]
where $H_k(t)$ denotes the Hermite polynomials.

 It is noteworthy that the dual space $\Sc_p^\prime(\R)$ is isometrically isomorphic to $\Sc_{-p}(\R)$ for $p \geq 0$. For $\phi \in \Sc_p(\R)$ and $\psi \in \Sc^\prime(\R)$, the duality pairing is denoted by $\langle \psi, \phi \rangle$. Furthermore, the following identities are established:
\[
\Sc^\prime(\R) = \bigcup_{p > 0} \left(\Sc_{-p}(\R), \| \cdot \|_{-p}\right),
\]

and

\[
\Sc_0(\R) = \mathcal{L}^2(\R).
\]

The fundamental inclusions among these spaces are as follows: for $0 < q < p$,

\[
\Sc(\R) \subset \Sc_p(\R) \subset \Sc_q(\R) \subset \mathcal{L}^2(\R) = \Sc_0(\R) \subset \Sc_{-q}(\R) \subset \Sc_{-p}(\R) \subset \Sc^\prime(\R).
\]

Consider the (partial) derivative operators $\partial: \Sc(\R) \to \Sc(\R)$. This operators extend to the dual space as $\partial: \Sc'(\R) \to \Sc'(\R)$, where for $\psi_1 \in \Sc'(\R)$, the extension is defined by:

\[
\langle \partial \psi_1, \psi_2 \rangle := -\langle \psi_1, \partial \psi_2 \rangle, \quad \text{for all } \psi_2 \in \Sc(\R).
\]

It is well-known that the operator $\partial$ maps $\Sc_p(\R)$ to $\Sc_{p-1}(\R)$, and the second-order derivatives $\partial^2$ map $\Sc_p(\R)$ to $\Sc_{p-1}(\R)$. These operators are bounded linear operators.

In our analysis, we use the following fact: if $f: \R \to \R$ is a smooth function with bounded derivatives (i.e., $f$ grows linearly), then the multiplication operator $M_f: \Sc(\R) \to \Sc(\R)$ defined by $M_f(\phi) := f \phi$ for all $\phi \in \Sc(\R)$ extends to $\Sc'(\R)$ as follows:
\[
\langle M_f \psi_1, \psi_2 \rangle := \langle \psi_1, M_f \psi_2 \rangle, \quad \text{for all } \psi_1 \in \Sc'(\R), \; \psi_2 \in \Sc(\R).
\]
For ease of notations, we may use $f$ instead of $M_f$. We also consider multiplication by  $x^\alpha,\;\alpha\in \N$. The operator $x^\alpha: \Sc(\R) \to \Sc(\R)$ is defined by $(x^\alpha \phi)(y) := x^\alpha \phi(y)$ for all $\phi \in \Sc(\R)$ and $y \in \R$. This operator extends by duality to $\Sc'(\R)$ as follows:
\[
\langle x^\alpha \psi_1, \psi_2 \rangle := \langle \psi_1, x^\alpha \psi_2 \rangle, \quad \text{for all } \psi_1 \in \Sc'(\R), \; \psi_2 \in \Sc(\R).
\]

 The distributional derivative operators $\partial: \Sc_p (\R)\to \Sc_{p}(\R)$ are densely defined closed linear unbounded operators. The following relations is well known \cite[Chapter 1]{MR1215939},
 \begin{equation}\label{derivative of h_n}
 \partial h _ n = \sqrt{\frac{n}{2}} h_{n-1} - \sqrt{\frac{n+1}{2}}h _ {n+1}\quad\text{and}\quad x h _ n = \sqrt{\frac{n}{2}} h_{n-1} + \sqrt{\frac{n+1}{2}}h _ {n+1}, \; \forall \;  n \geq 0
\end{equation}
  Consequently, we have $\partial: \Sc_p(\R) \to \Sc_{p-\frac{1}{2}}(\R)$ are bounded linear operators.
 
 Let $f: \R \to \R$ be a smooth function with bounded derivatives. Using above  relations, weak convergence and duality argument, it can be shown that the multiplication operator $f: \Sc_p (\R)\to \Sc_{p-1}(\R)$ are bounded linear operators (see \cite[Proposition 3.2]{MR2373102}). Using this observation, it follows that for $\phi \in \Sc_p(\R)$, $A^\ast(\phi)$ and $L^\ast(\phi)$ belong to $\Sc_{p-2}(\R)$. In fact,
\[
\| A ^\ast\phi \|_{p-2}^2 \leq C_1 \left\| \phi \right\|_p^2 \quad \text{and}\quad
\left\| L^\ast(\phi) \right\|_{p-2} \leq C_2 \left\| \phi \right\|_p,
\]
for some positive constants $C_1$ and $C_2$ (see \cite{MR2373102}).

\subsection{Countably Hilbertian Topology on the Complex valued Schwartz class functions}\label{S2:topology-complex}

We start by recalling the space of Complex-valued Schwartz class functions
$\Sc(\R; \C)$ \[
\Sc(\R^d; \mathbb{C}) = \left\{ \varphi \in C^\infty(\R^d; \mathbb{C}) : \forall k \geq 1, \max_{|\alpha| \leq k} \sup_{x \in \R^d} (1 + |x|^2)^k |\partial^\alpha \varphi(x)| < \infty \right\},
\]
and its dual $\Sc^\prime(\R; \C)$. A
countably Hilbertian topology on $\Sc(\R; \C)$ can be described (\cite{MR771478, MR1215939}), similar to the manner enacted for $\Sc(\R)$. For $p \in \R$, consider the norms $\| \cdot \|_{p, \C}$,
defined by the inner products
\[\inpr[p, \C]{f}{g}  = \sum_{n = 0}^\infty (2 n + d) ^ {2p}
\inpr[\mathcal{L}^2(\R; \C)]{f}{h _ n} \overline{\inpr[\mathcal{L}^2(\R;
\C)]{g}{h _ n}}, \; \forall f , g \in \Sc(\R; \C).\]
Here, the separable Hilbert spaces, the Hermite-Sobolev
spaces $\Sc_p(\R; \C), p \in \R$, are defined as the completion of $\Sc(\R; \C)$
in
$\norm[p, \C]{\cdot}$. Note that the dual space $\Sc_p^\prime(\R; \C)$ is
isometrically isomorphic with $\Sc_{-p}(\R; \C)$ for $p\geq 0$. For $\phi \in
\Sc(\R; \C)$ and $\psi \in \Sc^\prime(\R; \C)$, we write the duality action by
$\inpr{\psi}{\phi}$. We also have $\Sc(\R; \C) =
\bigcap_{p}(\Sc_p(\R; \C),\norm[p, \C]{\cdot}), \Sc^\prime(\R; \C) =
\bigcup_{p>0}(\Sc_{-p}(\R; \C),\norm[-p, \C]{\cdot})$ and $\Sc_0(\R; \C) =
\mathcal{L}^2(\R; \C)$. The following basic relations hold for the $\Sc_p(\R;
\C)$
spaces: for $0<q<p$, \[\Sc(\R; \C) \subset \Sc_p(\R; \C) \subset \Sc_q(\R; \C)
\subset
\mathcal L^2(\R; \C) = \Sc_0(\R; \C) \subset \Sc_{-q}(\R; \C)\subset
\Sc_{-p}(\R; \C)\subset \Sc^\prime(\R; \C).\]

Moreover, the following natural relations also hold:
\[\Sc(\R) \subset \Sc(\R; \C), \quad \Sc^\prime(\R) \subset \Sc^\prime(\R;
\C) \text{ and } \Sc_p(\R) \subset \Sc_p(\R; \C), \forall p \in \R,\]
with
\begin{equation}\label{real-complex-inner-pr-consistent}
  \inpr[p]{\phi}{\psi} = \inpr[p, \C]{\phi}{\psi}, \forall \phi, \psi \in
\Sc_p(\R).  
\end{equation}

Consider the derivative map denoted by $\partial:\Sc(\R; \C) \to
\Sc(\R; \C)$. This map extends by duality to
the distributional derivative operator $\partial: \Sc^\prime(\R; \C) \to \Sc^\prime(\R; \C)$ as follows: for
$\psi \in
\Sc'(\R; \C)$,
\[\inpr{\partial \psi}{\phi} := -\inpr{\psi}{\partial \phi}, \; \forall
\phi \in \Sc(\R; \C).\]
The regularity properties of $\partial$ remain the same as in the real-valued case. The analysis of multiplication of smooth functions with bounded derivatives also remain the same. We shall also use the multiplication by monomials $x^k: \Sc(\R; \C) \to \Sc(\R; \C), k \in \N$ in our arguments. To avoid repetition of the discussion in the real-valued case, we do not state the results again for the Complex-valued functions case. 

\subsection{Hermite Operator}\label{S:Hermite-Op}

We first recall well-known facts about the Hermite operator $\Hop$ from \cite[Chapter 1]{MR1215939} and \cite[Section 3]{MR1999259}. Define $\Hop : \Sc(\R) \to \Sc(\R)$ by 
\[(\Hop \phi)(x) := (x^2 - \partial^2) \phi (x), \forall \phi \in \Sc(\R), x \in \R.\]

\begin{proposition}[Properties of the Hermite operator, {\cite[Chapter 1]{MR1215939}} and {\cite[Section 3]{MR1999259}}]\label{Hop-properties}
We list a few facts about the Hermite operator.
\begin{enumerate}[label=(\roman*)]
\item $\Hop h_n = (2 n + 1) h_n$ for all $n \in \Z_+$.

\item Since $\{h_n : n \in \Z_+\}$ forms an orthonormal basis for $\mathcal{L}^2(\R)$, the above definition of $\Hop$ on $\Sc(\R)$ extends to $\Hop$ on $\mathcal{L}^2(\R)$ as a densely defined, closed, linear, positive operator. Using similar arguments, $\Hop$ extends linearly to $\Sc_p(\R), p \in \R$ and to $\Sc^\prime(\R)$.

\item For $p \in \R$, $p$-th power of $\Hop$ can be defined on first $\Sc(\R)$, and then extended to $\Sc^\prime(\R)$ as follows:
\[\Hop^p \phi := \sum_{n = 0}^\infty  (2n + 1)^{p} \inpr{\phi}{h_n} h_n.\]

\item For any $p, q \in \R$, $\|\Hop^p \phi\|_{q - p} = \|\phi\|_q, \forall \phi \in \Sc(\R)$ and hence, $\Hop^p: \Sc_q(\R) \to \Sc_{q - p}(\R)$ is a linear isometry. Moreover, this linear map is onto.

\item For any $p\in \R$, $\|\Hop^p \phi\|_{0} = \|\phi\|_p, \forall \phi \in \Sc_p(\R)$ and $\Sc_p(\R) = \Hop^{-p} \mathcal{L}^2(\R)$.
\end{enumerate}

\end{proposition}

The operator $\Hop$ on $\Sc(\R; \C)$ is defined similarly as
\[(\Hop \phi)(x) := (x^2 - \partial^2) \phi (x), \forall \phi \in \Sc(\R; \C), x \in \R.\]
Then, $\Hop$ is extended to $\Sc^\prime(\R; \C)$ and to $\Sc_p(\R; \C)$ in the natural way. Some well-known properties are listed below.

\begin{proposition}[{\cite[Section 3]{MR1999259}}]\label{Hop-C-properties}
\begin{enumerate}[label=(\roman*)]
    \item For $w \in \C$, $w$-th power of $\Hop$ can be defined on first $\Sc(\R; \C)$, and then extended to $\Sc^\prime(\R; \C)$ as follows:
\[\Hop^w \phi := \sum_{n = 0}^\infty  (2n+1)^{w} \inpr{\phi}{h_n} h_n.\]

\item For any $x, y \in \R$, $\Hop^{x+iy} = \Hop^x \Hop^{iy} = \Hop^{iy} \Hop^x$. Moreover, 
\[\norm[0,\C]{\Hop^{x+iy}\phi} = \norm[x,\C]{\phi}, \forall \phi \in \Sc_x(\R; \C).\]

\item For any $y \in \R$ and $p \in \R$, $\Hop^{iy}:\Sc_p(\R; \C) \to \Sc_p(\R; \C)$ is a linear isometry.
\end{enumerate}
\end{proposition}

We now highlight two simple observations in the next  lemma.

\begin{lemma}
Let $w \in \C$.
\begin{enumerate}[label=(\roman*)]
    \item For any $ \phi \in  \Sc  ( \R;\C ) $,
    \[  \overline {   \Hop ^ 
 w \phi } =   \Hop ^ { \Bar { w } } \Bar { \phi } = \sum _ n ( 2 n 
 + 1 ) ^ { \Bar { w } } \inpr[ 0 ,\C] { \Bar { \phi } } { h _ n } h _ n \]
 \item The adjoint of $  \Hop  ^ w $ is  $  \Hop ^ { \bar { w } } $.
\end{enumerate}
 
\end{lemma}

 \section{Proofs of Main results}\label{S2:proofs}
In this section, we discuss the proof of main results.
\subsection{Preliminary results}\label{Preliminary Results}
 First we need to  discuss some useful bounded operators.
\begin{lemma}\label{seq-of-order-1/n}
  For any non-zero complex number $w$,  consider the following sequences
  \begin{align*}
      \widetilde { \alpha } _ n ( w ) = & \sqrt { \frac { n } { 2 } } \left
[ \left ( \frac { 2 n - 1 } { 2 n + 1 } \right ) ^ w - 1 \right ] , \quad
 \widetilde { \beta } _ n ( w )  = \sqrt { \frac { n + 1 } { 2 } } \left [
1 - \left ( \frac { 2 n + 3 }
 {2 n + 1 } \right ) ^ w \right ]\\
\widetilde { \gamma } _ n ( w ) = &  \sqrt {  n ( n - 1 )} \left [ \left (
\frac { 2 n + 1 } { 2 n - 3 } \right ) ^ w - 1 \right ]
  , \quad         \widetilde { m } _ n ( w )  = \sqrt {  ( n + 1) ( n +2 )}
 \left [ \left ( \frac { 2 n + 1 } { 2 n + 5 } \right ) ^ w - 1 \right ]\\
   \widetilde { l } _ n ( w ) = & \frac {    2 n + 1  } { 2 } \left [ \left
( \frac { 2 n + 5 } { 2 n + 1 } \right ) ^ w - 1 \right ],\quad \widetilde { a  } _  n(w)  = \sqrt  { \frac { n  } {2  } }  \left [  1 -
\left ( \frac { 2  n +  1 }  { 2  n -  3 } \right ) ^  { w  } \right ]\\
\widetilde { b }
 _ n(w) =& \sqrt { \frac { n +  1 }  { 2 } } \left [ \left ( \frac { 2 n  - 1  } {
 2 n +  5 } \right ) ^  { w }  - 1  \right ].
  \end{align*}
Then the following inequalities are satisfied:
  \begin{align*}
      \left | \widetilde { \alpha } _ n ( w ) \right | \leq & \frac { M } {
\sqrt { n } } , &\left | \widetilde { \beta } _ n ( w ) \right | \leq &
\frac { M } { \sqrt { n } }, &\left | \widetilde { a } _ n ( w ) \right | \leq &
\frac { M } { \sqrt { n } }, &\left | \widetilde {a } _ n ( w ) \right | \leq &
\frac { M } { \sqrt { n } }\\
&\left | \widetilde { \gamma } _ n ( w ) \right | \leq 
 M , &\left | \widetilde { m } _ n ( w ) \right | \leq &
 M , &\left | \widetilde { l } _ n ( w ) \right | \leq &
 M,  & \forall  n \in \N
  \end{align*}
  for some $ M > 0 $. Consequently, all the sequences are bounded.
\end{lemma}

\begin{proof}
The argument is similar to \cite[proof of Lemma 2.2]{MR2590157}. Consider the analytic function  $ f : \D \to \C $  defined as
\[ f ( z ) = \left ( \frac { 2 - z } { 2 + z } \right ) ^  { w  } - 1 \] where $
\D = \{ z \in \C  \: ; \; | z | < 1 \} $ . Since $ f ( 0 ) = 0 $, therefore
there exists an analytic function $ g $ such that $ f (  z ) = z  g (  z ) , \;
\forall  z \in \D $. But on the compact set $ \overline { B ( 0 , \frac { 1 } {
2 } ) } $ the function $g$ is bounded, say by some constant $ R >  0 $.

Fix a positive integer $N$ such that $ \frac { 1  } {  N  } < \frac { 1 } {  2 }
$. Then $ \forall  n > N $,
\[ | \widetilde { \alpha } _ n ( w ) | = \sqrt { \frac { n } { 2 } } \left
| f \left ( \frac { 1  } { n } \right ) \right | \leq \frac { 1  } { \sqrt { 2 n
} } \left | g \left ( \frac { 1 } { n } \right ) \right | \leq \frac { R } {
\sqrt { 2 n }} . \]
Then taking $ M : = \sup \left \{ |\widetilde { \alpha } _ 1 ( w ) | ,
\sqrt { 2 } | \widetilde { \alpha } _ 2 ( w ) | ,\sqrt { 3 } |\widetilde {
\alpha } _ 3 ( w ) |,\frac { R } { \sqrt { 2 } } \right \} $ we have \[ |
\widetilde { \alpha } _ n ( w )| \leq \frac { M } { \sqrt { n } } \;
\forall n > 0 \]
From this inequality, required bound can be  obtained. Other proofs  are
similar.
\end{proof}

In an analogous manner, the next result follows. We skip the proof for brevity.
  \begin{lemma}\label{seq-of-order-1/n^2}
  For any non-zero complex number $w$,   the following  sequences
  \begin{align*}
       a  _ n = &  \sqrt { \frac { n ( n - 1 ) ( n - 2 ) ( n - 3 )   } { 1 6 }
}  \left [ \left ( \frac{ 2 n + 1 } { 2 n - 7 } \right ) ^ w - 2  \left (
\frac{ 2 n - 3 } { 2 n - 7 } \right ) ^ w + 1 \right ]\\
       b _ n = &  \frac {  n ( n - 1 )}{ 4 } \left [ \left ( \frac { 2 n - 3 } {
2 n + 1 } \right ) ^ w +  \left (\frac { 2 n + 5 } { 2 n + 1 } \right ) ^ w - 2
\right ]  \\
       c _ n = &  \sqrt { \frac { n ( n + 1 ) ( n + 2 ) ( n + 3 )   } { 1 6 }  }
 \left [ \left ( \frac { 2 n + 1 } { 2 n  +9  } \right ) ^ w  - 2 \left (
\frac { 2 n + 5 } { 2 n  +9  } \right ) ^ w  + 1 \right ]
  \end{align*}
 are bounded and the following sequences
 \begin{align*}
      l  _ n = &  \sqrt { \frac { n ( n - 1 ) ( n - 2 )    } { 8 }  }  \left [
\left ( \frac{ 2 n + 1 } { 2 n - 5 } \right ) ^ w  -  \left ( \frac{ 2 n - 3 } {
2 n - 5 }  \right ) ^ w - \left ( \frac{ 2 n - 1 } { 2 n - 5 } \right ) ^ w + 1
\right ]\\
 m  _ n = & \frac{ n + 1 }{ 2 } \sqrt {  \frac { n  }{ 2 } } \left [ -\left (
\frac { 2 n + 1 } { 2 n - 1 } \right ) ^ w + \left ( \frac { 2 n - 3 } { 2 n - 1
} \right ) ^ w + \left ( \frac { 2 n + 3 } { 2 n - 1 } \right ) ^ w - 1 \right
]\\
  t   _ n = & \frac{ n  }{ 2 } \sqrt {  \frac { n + 1 }{ 2 } } \left [   \left (
\frac { 2 n + 5 } { 2 n + 3 } \right ) ^ w  - \left ( \frac { 2 n + 1 } { 2 n +
3 } \right ) ^ w - \left ( \frac { 2 n - 1 } { 2 n + 3 } \right ) ^ w + 1 \right
]\\
\alpha _ n  = & \sqrt { \frac { ( n + 1 ) ( n + 2 ) ( n + 3 )   } { 8 }  }
\left [ \left ( \frac { 2 n + 1 } { 2 n  + 7  } \right ) ^ w - \left ( \frac { 2
n + 3 } { 2 n  + 7  } \right ) ^ w -  \left ( \frac { 2 n + 5 } { 2 n  + 7  }
\right ) ^ w  + 1 \right ]\\
   \beta _ n = & \sqrt {  \frac { n  }{ 2 } } \left [  \frac{1}{2} \left ( \frac
{ 2 n + 1 } { 2 n - 1 } \right ) ^ w +\left ( \frac { 2 n - 3 } { 2 n - 1 }
\right ) ^ w - \frac{3}{2} \right ]\\
 \gamma _ n = &  \sqrt {  \frac { n + 1 }{ 2 } }  \left [ \frac{1}{2} \left (
\frac { 2 n + 1 } { 2 n + 3 } \right ) ^ w + \left ( \frac { 2 n + 5 } { 2 n + 3
} \right ) ^ w - \frac{3}{2} \right ]
 \end{align*}
   satisfy the following inequalities
\begin{align*}
      \left |  l  _ n \right | \leq & \frac { M } { \sqrt { n } } , &\left |  m
_ n \right | \leq & \frac { M } { \sqrt { n } } &\left |  t  _ n \right | \leq &
\frac { M } { \sqrt { n } } &\left |  \alpha  _ n \right | \leq &
\frac { M } { \sqrt { n } } &\left |  \beta  _ n \right | \leq &
\frac { M } {  n^\frac{3}{2} } &\left |  \gamma  _ n \right | \leq &
\frac { M } { n^\frac{3}{2} } & \forall  n \in \N
  \end{align*}
hold for some $M > 0$.
\end{lemma}

Fix $w \in \C$ and $k \in \N$. We start by defining linear operators  $ T_{\Tilde{\alpha}(w)} , T_{\Tilde{\beta}(w)} , U _ { + k } , U _ { - k } $   on $ \Sc (\R; \C ) $ via the formal expressions for $ \phi = \sum _ { n = 0 } ^ { \infty } \phi _ n h  _ n  \in  \Sc (\R; \C ) $ as follows,
\begin{align*}
    T_{ \Tilde{ \alpha }(w) } \phi = & \sum _ { n = 0 } ^ { \infty } \Tilde{\alpha} _ n ( w ) \phi _ n h _ n     & T_{\Tilde{\beta}(w)} \phi = \sum _ { n = 0 } ^ { \infty } \Tilde{\beta} _ n  ( w ) \phi _ n h _ n  \\
    U _ { + k } \phi = & \sum _ { n =  0 } ^ { \infty } \phi _ { n +  k } h _ n   
    & U _ { - k}  \phi = \sum _ { n = 0 }  ^ { \infty } \phi _ { n - k } h _ n,   
\end{align*}
where $\Tilde{\alpha} _ n$ and $\Tilde{\beta} _ n$ are as in Lemma \ref{seq-of-order-1/n}.

\begin{lemma}\label{bnd-scale-shift-ops}
    $ T_{\Tilde{\alpha}( w )} , T_{\Tilde{\beta} ( w )} ,   U  _ { + k } ,    U _ { - k } $ are bounded linear operators on $ \left ( \Sc (\R; \C ) , \| \cdot \| _ {q, \C} \right ) , $ and hence can be extended to bounded linear operators on $ \left ( \Sc _ q (\R; \C ) , \| \cdot \| _ {q, \C} \right ) $.
\end{lemma}

\begin{proof}
Boundedness of the scaling operators
$ T_{\Tilde{\alpha}( 2q )} , T_{\Tilde{\beta} ( 2q )}$ and the shift operators $U  _ {  +  1 } ,  U _  {  - 1  } $ on $ \left ( \Sc (\R; \C ) , \| \cdot \| _ {q, \C} \right ) $ and then to $ \left ( \Sc_q (\R; \C ) , \| \cdot \| _ {q, \C} \right ) $ follow from standard arguments (see, for example, the proof of \cite[Theorem 2.1]{MR3331916}).
\end{proof}

The next result is a simple extension of \cite[Theorem 2.1]{MR3331916}, in identifying the adjoint of the distributional derivative operators on the Hermite-Sobolev spaces.

\begin{lemma}\label{adj-partial-on-complex}
  Fix $q \in \R$. We have, for any $ \phi , \psi \in  \Sc(\R; \C) $,
  \[ { \inpr { \phi } { \partial \psi } } _ {q, \C} + { \inpr { \partial \phi } { \psi } } _ {q, \C} = { \inpr { \phi } { ( T_{\Tilde{\alpha} ( 2q )} U _ { - 1 } + T_{\Tilde{\beta} ( 2q ) } U _ { + 1 } ) \psi } } _ {q, \C} \]
  and hence we obtain the adjoint operator
  \[ \partial ^ \ast = - \partial + T  \; o n \;  \Sc (\R; \C ) \]
  where $ T  =  T_{\Tilde{\alpha} ( 2q )} U _ { - 1 } + T_{\Tilde{\beta} ( 2q ) } U _ { + 1 } $ is a bounded linear operator on $ \left ( \Sc _ q (\R; \C ) , \| \cdot \| _ {q, \C} \right ) $.
\end{lemma}
\begin{proof}
Our arguments are similar to those in \cite[Theorem 2.1]{MR3331916}. Now, for any \\
$ \phi = \sum _ { n = 0 } ^ { \infty } \phi _ n h  _ n, \psi = \sum _ { n = 0 } ^ { \infty } \psi _ n h  _ n \in  \Sc(\R; \C) $, using \eqref{derivative of h_n} we have,
\[\partial \phi = \sum ^ \infty _ { n = 0 } \phi _ n ( \partial h _ n ) 
    =  \sum ^ \infty _ { n  =  0 } \sqrt { \frac { n  + 1  }  {  2 } } \phi _ { n + 1 } h _ n - \sum ^ \infty _ { n = 0  } \sqrt { \frac { n  } { 2  } } \phi _ { n  - 1 }  h _  n.\]
Similarly \[ \partial \psi = \sum ^ \infty _ { n =  0 } \sqrt { \frac { n + 1 } { 2 } } \psi _ { n + 1 } h _ n -  \sum ^ \infty _ { n  = 0} \sqrt { \frac { n } { 2  } } \psi _ { n- 1 } h _  n.\]
Therefore 
\begin{align*}
     \inpr[q, \C] { \phi } { \partial \psi }  = & \sum ^ \infty _ { n = 0 } ( 2 n  + 1 ) ^ { 2 q  } { \inpr[q, \C] { \phi } { h _ n } }   \overline { { \inpr[q, \C] { \partial \psi } { h _ n  } } } \\
    =& \sum ^ \infty _ { n =  0 } ( 2 n +  1 ) ^ { 2  q }  \phi _ n \left [ \sqrt { \frac { n  + 1 }  { 2  } } \Bar { \psi } _ {  n +  1 }  - \sqrt { \frac { n } 
 { 2 } } \Bar { \psi } _  { n  - 1  } \right ] 
\end{align*}
and
\begin{align*}
    \inpr [ q,\C ] { \partial \phi } { \psi } = & \sum ^ \infty _ { n = 0 }( 2 n + 1 ) ^ { 2 q } \bar { \psi } _  n \left [ \sqrt { \frac { n + 1 } { 2 } } \phi _ { n + 1 } - \sqrt { \frac { n } { 2 } } \phi _ { n - 1 } \right ] \\
    = & \sum ^ \infty _ { n = 0 } ( 2 n + 1 ) ^ { 2 q } \bar { \psi } _ n \sqrt { \frac { n + 1 } { 2 } } \phi _ { n + 1 } - \sum ^ \infty _ { n = 0 }( 2 n + 1 ) ^ { 2 q } \bar { \psi } _ n \sqrt { \frac { n } { 2 } } \phi _ { n - 1 } \\
     = & \sum ^ \infty _ { n = 1 } ( 2 n - 1 ) ^ { 2 q } \bar { \psi } _ { n - 1 } \sqrt { \frac { n } { 2 } } \phi _ { n } - \sum ^ \infty _ { n = - 1 }( 2 n + 3 ) ^ { 2 q } \bar { \psi } _ { n + 1 } \sqrt { \frac { n + 1 } { 2 } } \phi _ { n } \\
    = & \sum ^ \infty _ { n = 0 } \sqrt { \frac { n } { 2 } } ( 2 n - 1 ) ^ { 2 q } \bar { \psi } _ { n - 1 } \phi _ { n }  - \sum ^ \infty _ { n  = 0 }  \sqrt { \frac { n + 1 } { 2 } } ( 2 n + 3 ) ^ { 2 q } \bar { \psi } _ { n + 1 } \phi _ { n } \\
     = & \sum ^ \infty _ { n = 0 }  ( 2 n + 1 ) ^ { 2 q } \phi _ { n  }\bar { \psi } _ { 
 n - 1 } \left [ \sqrt { \frac { n } { 2 } } \left ( \frac { 2 n - 1 } 
 { 2 n + 1 } \right ) ^ { 2 q } \right ]\\
-&\sum ^ \infty _ { n = 0 } ( 2 n + 1 ) ^ { 2 q } \phi _ { n } \bar { \psi } _ { n + 1 } \left [ \sqrt { \frac { n + 1 } { 2 } } \left ( \frac { 2 n +  3 } { 2 n + 1 } \right ) ^ {  2q } \right ].
    \end{align*}
Combining the above two expressions, we end up with
\begin{align*}
        \inpr [ q, \C ] { \phi } { \partial \psi } + \inpr [ q, \C ] { \partial \phi } {  \psi }  = & \sum ^ \infty _ { n = 0 } ( 2 n + 1 ) ^ {  2 q } \phi _ { n } \bar{ \psi } 
 _ { n -  1 } \left [ \sqrt { \frac{ n }  { 2 }  } \left ( \frac  { 2 n  - 1 } { 
 2n  + 1 } \right ) ^ {  2 q } - 1 \right ] \\
         & - \sum ^ \infty _ { n  = 0 } (  2 n + 1  ) ^ { 2 q 
 } \phi _ { n } \bar { \psi } _ { n + 1 } \left [ 1 - \sqrt { \frac { n + 1 } { 2 } } \left ( \frac { 2 n + 3 } { 2 n + 1 } \right ) ^ { 2 q } \right ] \\
        = & \sum ^ \infty _ { n = 0 } ( 2 n + 1 ) ^ { 2 q } \phi _ { n } \bar { \psi } _ { n - 1 } \Tilde{\alpha}  _ n - \sum ^ \infty _ { n =  0 } ( 2 n + 1 ) ^ { 2 q } \phi _ { 
 n } \bar { \psi } _ { n + 1 }  \Tilde{\beta} _ n \\
        = &  \inpr [ q, \C ] { \phi } { T_{\Tilde{\alpha} ( 2q )} U _ { - 1 } \psi } +  \inpr  [ q, \C ] { \phi } { T_{\Tilde{\beta} ( 2q ) } U _ { + 1 } \psi } \\
        = & \inpr [ q, \C ] { \phi } { \left ( T_{\Tilde{\alpha} ( 2q )} U _ { - 1 } + T_{\Tilde{\beta} ( 2q ) } U _ { + 1 } \right ) \psi }
        \end{align*}
    Consequently, 
    \[ \partial ^ \ast = - \partial + T   \]
    where $ T  =  T_{\Tilde{\alpha} ( 2q )} U _ { - 1 } + T_{\Tilde{\beta} ( 2q ) } U _ { + 1 } $  which is a bounded linear operator on on $ \left ( \Sc _ q (\R; \C ) , \| \cdot \| _ {q, \C} \right ) $.
\end{proof}

Next, we establish the boundedness of several operators in the next few lemmas, from Lemma \ref{bndop1}
 to Lemma \ref{A12L3}.

 \begin{lemma}\label{bndop1}
 Fix $ w \in \C $ and $ q \in \R$. Then,  $  \Hop ^ { w } \partial    \Hop  ^ {
-w } - \partial   $ is a bounded linear operator on $ \left ( \Sc _ q(\R; \C)
, \| \cdot \| _  {q, \C} \right ) $.
\end{lemma}
\begin{proof}
For any $\phi \in \Sc(\R; \C)$,
  \begin{align*}
   \left (   \Hop^  { w } \partial   \Hop  ^ {  - w  } -  \partial \right ) \phi = &
\sum _ n ( 2  n + 1  ) ^ { 2  q }  \inpr [ 0,\C ] { \left (   \Hop^  { w } \partial   \Hop
^ {  - w  } -  \partial \right ) \phi } { h _ n } h _ n \\
    = & \sum _ n ( 2 n +  1 ) ^ { 2 q } \left (  \inpr[0, \C] {  \Hop ^ { w } \partial  \Hop ^
{ - w } \phi } { h _ n } -  \inpr [ 0,\C ] { \partial \phi } { h _ n } \right ) h _
n \\
    = & \sum _ n ( 2 n  + 1 ) ^ { 2 q  } \left ( (  2 n +  1)  ^ { w  }
\inpr[0,\C]{ \partial  \Hop ^ {  - w } \phi }  { h _  n }    -  \inpr [ 0,\C ] {
\partial\phi }{ h _
 n  }  \right ) h _  n
 \end{align*}
 Now, using \eqref{derivative of h_n}
 \begin{align*}
     &( 2 n +  1 ) ^ w \inpr [ 0,\C ] { \partial  \Hop ^ { - w } \phi } { h _ n }  -
\inpr [ 0,\C ] { \partial \phi } { h _ n }\\ 
= & - ( 2 n  + 1 ) ^  w  \inpr [ 0,\C ] {
 \Hop ^ { - w } \phi } { \partial h _ n } +  \inpr [ 0,\C ] { \phi } { \partial h _ n }
\\
     = & -( 2 n + 1 ) ^ w\inpr [ 0,\C ] {   \Hop ^ { - w } \phi } { \sqrt { \frac { n }
 { 2 } } h _ { n - 1 } + \sqrt { \frac { n +  1 } { 2 }  } h _ { n + 1  } } \\
     & +  \inpr [ 0,\C ]  { \phi } { \sqrt { \frac { n } { 2 } } h _ {  n - 1 } +
\sqrt { \frac { n
 + 1 }  { 2 }  } h _  { n  + 1  } } \\
      = & -( 2 n +  1 ) ^ { w }  \inpr [ 0,\C ] {  \phi } { \sqrt { \frac { n }  {
2 }  } (
 2 n  - 1 ) ^ { -\bar { w  } } h  _  { n  - 1  }  + \sqrt { \frac { n  + 1 }  {
2 }  } (
 2 n  + 3  ) ^ { - \bar { w }  } h _  { n +   1  } } \\
 & +   \inpr [ 0,\C ] { \phi } { \sqrt { \frac { n } { 2 } } h _ { n - 1 } + \sqrt
{ \frac { n + 1 } { 2 } } h _ { n + 1 } } \\
      = & -\sqrt { \frac { n  } {  2 }  } \left ( \left ( \frac{ 2  n + 1  }  {
2  n - 1
 } \right ) ^ w - 1 \right ) \inpr [ 0,\C ]  { \phi } {  h _  { n  - 1  }  } \\
      & - \sqrt { \frac { n + 1 } { 2 } } \left ( 1 - \left ( \frac { 2 n + 1 }
{ 2 n + 3 } \right ) ^ { w } \right ) \inpr [ 0,\C ] { \phi } { h _ { n + 1 } } \\
      = &-\left ( \widetilde { \alpha } _ n ( - w ) \phi _ { n  - 1 } +
\widetilde {\beta }  _ n ( - w ) \phi _  { n  + 1 } \right ),
 \end{align*}
where $\Tilde{\alpha} _ n$ and $\Tilde{\beta} _ n$ are as in Lemma \ref{seq-of-order-1/n}. Consequently, from our above computations and using Lemma \ref{bnd-scale-shift-ops},
\begin{align*}
     \left (  \Hop ^ { w } \partial  \Hop  ^ { - w  }  -\partial \right ) \phi = & -\sum
_ n
 ( 2 n  + 1 )  ^ {  2 q  } \left ( \widetilde { \alpha } _ n (-w) \phi _ {
n - 1
 } + \widetilde { \beta } _ n (-w) _ n \phi _ {  n +  1 } \right ) h _ n\\
= & -\sum _ n  \widetilde { \alpha } _ n (-w) ( 2  n + 1 )  ^ { 2  q  }
\phi _ {  n -
 1 }  h _  n + \sum _ n \widetilde { \beta } _ n (-w) ( 2  n  + 1  ) ^ {  2
q } \phi _  { n +
 1 } h  _ n   \\
= &- \left (  T _ {\widetilde { \alpha } ( -w )}  U _ {  - 1 }  +  T _
{\widetilde { \beta } ( -w )}  U  _ { + 1 } \right ) \phi.
 \end{align*}
 The result follows by density arguments.
  \end{proof}

  The next result is an analogue of \cite[Lemma 2.4]{MR3331916}.
 \begin{lemma}\label{bndblm1}
  Fix $ w \in \C $ and $ q \in \R$. The map $ \inpr [ q, \C ] { \partial ( \cdot ) } { \left (  \Hop ^  { w } \partial  \Hop
^ { -  w } -\partial \right ) ( \cdot ) } : \Sc (\R; \C ) \times \Sc (\R; \C ) \mapsto
\C $ defined by \[ ( \phi ,\psi ) \to  \inpr [ q, \C ] { \partial \phi } { \left
(  \Hop  ^ { w }  \partial  \Hop  ^ {  -  w }  - \partial \right ) \psi } \; \forall
\phi , \psi \in \Sc (\R; \C ) \]
  is a bounded bilinear form in $ \| \cdot \| _ {q, \C} $ and hence extends to a
bounded bilinear form on $ \left ( \Sc _ q (\R; \C ) , \| \cdot \| _ {q, \C} \right )
\times \left ( \Sc _ q (\R; \C ) , \| \cdot \| _ {q, \C} \right ) $.
\end{lemma}

 \begin{proof}
  For $ \phi , \psi \in \Sc (\R; \C ) $,
  \begin{align*}
      &\inpr [ q, \C ] { \partial \phi } { \left (  \Hop ^ { w }  \partial  \Hop ^ { - w }
-\partial \right ) \psi }\\
= & \sum ^ \infty _ {  n =  0 }  ( 2 n  + 1
 ) ^  { 2 q  } \inpr [ 0,\C ] {  \partial \phi } {  h _ n }  \overline { \inpr [ 0,\C
]  { \left (  \Hop ^
 { w } \partial  \Hop ^ {  - w  } - \partial \right ) \psi } { h  _ n } } \\
      = & \sum ^ \infty _ { n = 0 } ( 2 n + 1  ) ^ { 2 q } \inpr [ 0,\C ] { \phi }
{\partial h _ n }
      \overline { \inpr [ 0,\C ] { ( T _ {\widetilde { \alpha } ( -w )}  U _ {
 - 1 }  +  T _ {\widetilde { \beta } ( -w )}  U  _ { + 1 } ) \psi } { h  _
n  } } \\
       = &  \sum ^ \infty _ { n = 0 } ( 2 n + 1 ) ^ { 2 q } \inpr [ 0,\C ] { \phi }
{ \sqrt { \frac { n  } {  2 } }  h _ {  n -  1 } + \sqrt { \frac { n  + 1 }
 { 2  }  } h  _ {  n + 1  } } \overline { ( \widetilde { \alpha } _ n (-w)
\psi _ { n - 1 } + \widetilde { \beta } _ n (-w) \psi _ {  n + 1 } ) } \\
       = &  \sum ^ \infty _ {  n =  0 }  ( 2 n  + 1 )  ^ { 2 q } \left ( \sqrt {
\frac { n } { 2 } } \phi _ {  n - 1  } - \sqrt { \frac { n  + 1  } {  2 } } \phi
_ {
 n + 1 } \right ) ( \bar { \widetilde { \alpha }} _ n (-w)  \bar { \psi } _
{ n - 1 } + \bar {\widetilde { \beta }} _ n (-w) \bar { \psi } _ { n+ 1 }
)
  \end{align*}
  From  Lemma \ref{seq-of-order-1/n}, we have $ \widetilde { a } _  n(w) ,
\widetilde { b  } _  n (w)$ have the order $ \frac { 1  } { \sqrt { n  } } $. Now
using Cauchy-Schwarz inequality, we get  $ C > 0 $ such that \[ \left | { \inpr
{ \partial \phi } { \left (  \Hop ^ { w } \partial  \Hop ^ { - w } -  \partial \right )
\psi } } _ {q, \C} \right | \leq C \| \phi \| _ {q, \C} \| \psi \| _ {q, \C}.\]
The result follows by density arguments.
  \end{proof}

\begin{lemma}
Fix $w  \in \C $ and $q\in \R$. Then, 
\begin{enumerate}[label=(\roman*)]
    \item\label{Mix-op-1}   $ 2 \left   (  \Hop ^  w    \partial ^2 x ^ 2  \Hop  ^ { -  w  } -  \partial ^2 x ^
2 \right ) - 2 x \partial \left   (  \Hop ^  w    \partial x  \Hop  ^ { -  w  } -
\partial x \right )   $  is a bounded linear operator on $ \left ( \Sc _ q(\R;\C) ,\norm[q,\C]{ \cdot} \right ) $.
\item\label{A12L3}
    $ \left (  \Hop ^ { w } \partial ^
2 x   \Hop  ^ { - w } - \partial ^ 2 x     -x\partial \left (  \Hop ^ { w } \partial  \Hop ^ { - w } - \partial \right )  - \partial \left (  \Hop ^ {
w } \partial x  \Hop ^ { - w } - \partial x \right ) \right ) $
is a bounded linear operator   on $ \left ( \Sc _ q(\R;\C) ,\norm[q,\C]{ \cdot} \right ) $.
\item\label{bndL2}
   $  \Hop ^ { w }  \partial ^ 2  \Hop ^ { - w }
-  \partial ^ 2 $ is a bounded linear operator on $ \left ( \Sc _ q(\R;\C) ,\norm[q,\C]{\cdot } \right ) $.
\item\label{bndop2}
$  \Hop ^ {  w  } x \partial  \Hop ^ { -
w  } - x \partial $ is a bounded linear operator on $ \left ( \Sc _ q(\R; \C) ,\| \cdot \| _ {q, \C} \right ) $.
\end{enumerate}
 
   \end{lemma}

Now, we are ready to prove one of the main results, which is an extension of \cite[Theorem 4.2]{MR3331916}).

\subsection{First Approach }\label{First approach}
 
\begin{proof}[Proof of Theorem \ref{Mono-ineq-affine-Form}]
First, we work with Complex-valued functions and then, we restrict our attention to the real-valued case.

The pair of operators $(L ^ \ast, A ^ \ast) $ have the following form
\[ L^\ast \phi = \frac{ 1 }{ 2 }  \partial ^ 2 \left ( (\alpha x + \beta )^2
\phi \right ) - \partial \left ( ( \gamma x + \delta  ) \phi \right ), \quad A ^
\ast \phi = - \partial \left ( (\alpha x + \beta ) \phi\right ), \quad \phi \in
\Sc(\R;\C).\]
Fix $\phi \in
\Sc(\R;\C)$. Motivated by the approach in \cite[Theorem 2.1]{MR1999259}, we consider the complex valued function \[ F ( w ) = 2 \inpr [ 0,\C ] { \phi }{  \Hop
^ { \Bar { w } } L ^ \ast   \Hop ^ { - \Bar { w } } \bar { \phi } }   + \inpr [ 0,\C ]
 {  \Hop ^ { w } A ^ \ast   \Hop ^ { - w } \phi } {  \Hop ^ { \Bar{ w }  } A ^ \ast   \Hop ^ { -
\Bar{ w } } \bar { \phi } }, \forall w \in \C.\]
We expand $F(w)$ in terms of bounded operators for every fixed $w$. Consequently, the required Monotonicity inequality \eqref{Monotoniticity-inequality-chap7} follows as a special case by taking $w = p$.

The above expression for $F(w)$ is simplified as
\begin{align*}
     F ( w ) = &  2\inpr [ 0,\C ] { \phi } {  \Hop ^ { \bar { w } }  L^\ast   \Hop
 ^ { - \bar { w } } \phi   }  + \inpr [ 0,\C ] {  \Hop ^ { w } A ^ \ast   \Hop ^ { - w }
\phi  } {  \Hop ^ { \Bar{ w } } A ^ \ast   \Hop ^ { - \Bar { w } } \phi ) } \\
 = & \left (  2\inpr [ 0,\C ] { \phi } {  L ^ \ast  \phi   }  + \inpr [ 0,\C ] {  A ^
\ast   \phi  } { A ^ \ast   \phi ) } \right ) +  \inpr [ 0,\C ] { \left (  \Hop ^ { w }
A ^ \ast   \Hop ^ { - w } - A ^ \ast  \right )  \phi  } {  \left (  \Hop ^ { \Bar{ w } }
A ^ \ast   \Hop ^ { - \Bar { w } } - A ^ \ast  \right )  \phi ) } \\
 & + 2 \left (  \inpr [ 0,\C ] { \phi } { \left (  \Hop ^ { \bar { w } }  L ^ \ast   \Hop
 ^ { - \bar { w } } -  L ^ \ast  \right ) \phi   }  +   \inpr [ 0,\C ] {  A ^ \ast
\phi  } {\left (  \Hop ^ { \Bar{ w } } A ^ \ast  \Hop ^ { - \Bar { w } } - A ^ \ast
\right ) \phi ) }  \right )\\
 = & \left (  2\inpr [ 0,\C ] { \phi } {  L ^ \ast  \phi   }  + \inpr [ 0,\C ] {  A ^
\ast  \phi  } { A ^ \ast   \phi ) } \right ) +  \inpr [ 0,\C ] { \left (  \Hop ^ { w }
A ^ \ast   \Hop ^ { - w } - A ^ \ast  \right )  \phi  } {  \left (  \Hop ^ { \Bar{ w } }
A ^ \ast   \Hop ^ { - \Bar { w } } - A ^ \ast  \right )  \phi ) } \\
 & +   \inpr [ 0,\C ] { \phi } {  \left ( 2 \left (  \Hop ^ { \bar { w } }  L ^ \ast
 \Hop
 ^ { - \bar { w } } -  L ^ \ast  \right ) + 2 A  \left (  \Hop ^ { \Bar{ w } } A ^
\ast   \Hop ^ { - \Bar { w } } - A ^ \ast    \right )   \right )   \phi   }.
\end{align*}
However, $\forall w \in \C$
\begin{align*}
 &\left (  \left   (  \Hop ^  w    L ^ \ast   \Hop  ^ { -  w  } -  L ^ \ast  \right ) + A
 \left   (  \Hop ^  w    A ^ \ast   \Hop  ^ { -  w  } -  A ^ \ast  \right )\right ) \phi\\
= & \alpha ^ 2 \left (   \left   (  \Hop ^  w     \frac{ 1 }{ 2 } \partial ^ 2 x ^ 2
 \Hop  ^ { -  w  } -   \frac{ 1 }{ 2 } \partial ^ 2 x ^ 2 \right ) +  A _ 1 \left
(  \Hop ^  w    A ^ \ast _1  \Hop  ^ { -  w  } -  A ^ \ast _ 1 \right ) \right ) \phi\\
  +& \alpha \beta (  \left   (  \Hop ^  w   \frac{ 1 }{ 2 } \partial ^ 2 x  \Hop  ^ { -
w  } - \frac{ 1 }{ 2 } \partial ^ 2 x \right )  +  A _ 1 \left   (  \Hop ^  w    A
^ \ast _2  \Hop  ^ { -  w  } -  A ^ \ast _ 2 \right ) + A _ 2 \left   (  \Hop ^  w    A
^ \ast _1  \Hop  ^ { -  w  } -  A ^ \ast _ 1 \right ) ) \phi\\
  +& \beta  ^ 2 \left(  \left   (  \Hop ^  w    \frac{1}{2}\partial ^2  \Hop  ^ { -  w
} -  \frac{1}{2}\partial ^2 \right ) +  A _ 2 \left   (  \Hop ^  w    A ^ \ast _2  \Hop
^ { -  w  } -  A ^ \ast _ 2 \right ) \right ) \phi\\
 +& \alpha \left (  \Hop ^ w A ^ \ast _ 1   \Hop ^ {-w} -A ^ \ast _ 1  \right )  \phi
+ \beta \left (  \Hop ^ w A ^ \ast _ 2      \Hop ^ {-w}   -  A  ^ \ast _ 2    \right )
\phi,
\end{align*}
 where \[\quad L ^ \ast _ 3 \phi = \frac{ 1 }{ 2 } \partial ^ 2 x
\phi,\quad A ^ \ast _ 1 \phi = - \partial ( x \phi) ,\quad A ^ \ast _ 2 \phi =
- \partial \phi, \forall \phi \in \Sc(\R; \C).\]
Consequently, the relevant bound follows by Cauchy-Schwarz inequality and the lemmas Lemma \ref{bndop1}
 to Lemma \ref{A12L3}.
 \end{proof}

\subsection{Second Approach}\label{second approach}

\begin{lemma}\label{Mono-ineq-gen-functin-P012}
    Let $\sigma, \;  b $ be real valued smooth functions with bounded
derivatives. Recall the operators $ A ^ * , L ^ * $ given by
\begin{align*}
          A ^ \ast \phi = & - \partial ( \sigma \phi ) &   L ^ \ast \phi = &
\frac { 1 }
 { 2 }  \partial ^ 2
( \sigma ^ 2 \phi ) - \partial ( b \phi ).
     \end{align*}
     Then, for $ p = 0 , 1$ and  $ 2 $, the Monotonicity inequality holds. i.e.
     \[ 2 \inpr [ p ]  {  \phi } { L ^ \ast \phi } + \inpr [ p ] {  A ^ \ast
\phi } { A ^ \ast \phi } \leq C \| \phi \| ^ 2 _ p , \quad \forall \phi \in \Sc(\R).\]
\end{lemma}
\begin{proof}
The case $p = 0$ is known (see, for example, \cite{MR3331916}) and follows from an integration by parts argument. Now, consider the case $p = 1$.

Note that $2 \inpr [ 1 ]  {  \phi } { L ^ \ast \phi } + \inpr [ 1 ] {  A ^ \ast
\phi } { A ^ \ast \phi } = 2 \inpr [ 0 ]  {\Hop  \phi } {\Hop L ^ \ast \phi } + \inpr [ 0 ] {\Hop  A ^ \ast
\phi } {\Hop A ^ \ast \phi }$. We have \[ \Hop  A ^ *  \phi = \left ( \sigma x ^ 2 \partial \phi - \sigma
\partial ^ 3 \phi \right ) + \left ( \partial \sigma .x ^ 2 . \phi - 3 \partial
\sigma . \partial ^ 2 \phi \right ) - 3 \partial ^ 2 \sigma \partial \phi -
\partial ^ 3 \sigma.\phi, \]
    and hence
    \begin{align*}
        \inpr [ 0 ] {   \Hop A ^ * \phi } {  \Hop A ^ *   { \phi } } &=   \inpr [ 0 ]
 {  \sigma x ^ 2 \partial \phi - \sigma \partial ^ 3 \phi } {  \sigma x ^ 2
\partial  { \phi }  - \sigma \partial ^ 3   { \phi } } + 2 \inpr [ 0 ] {
\sigma x ^ 2 \partial \phi - \sigma \partial ^ 3 \phi } { \partial \sigma x ^ 2
 { \phi }  - 3 \partial \sigma  \partial ^ 2   { \phi } }\\
        &- 2 \inpr [ 0 ] { \sigma x ^ 2  \partial \phi - \sigma \partial ^ 3
\phi } { 3 \partial ^ 2 \sigma \partial   { \phi } } + FSGT
    \end{align*}
We focus on the first three terms. Note that
\begin{enumerate}[label=(\roman*)]
    \item \begin{align*}
          &\inpr [  0 ] {  \sigma x ^ 2 \partial \phi - \sigma  \partial ^ 3 \phi
} { \sigma x ^ 2 \partial   { \phi }  - \sigma \partial ^ 3   { \phi } }\\
&=  \left \| \sigma x ^ 2 \partial \phi \right \| ^ 2 _ 0 + \left\| \sigma
\partial ^ 3 \phi \right\| ^ 2 _ 0  - 2 \inpr [ 0 ] {  \sigma x ^ 2  \partial
\phi  } { \sigma \partial ^ 3   {  \phi } }\\
          &=  \left\| \sigma  x ^ 2 \partial \phi \right \| ^ 2 _ 0 + \left \|
\sigma \partial ^ 3 \phi \right\| ^ 2 _ 0 - 2 \inpr [ 0 ] { \sigma \partial ^ 2
\sigma \partial \phi  } { x ^ 2 \partial
    { \phi }  } + 2 \inpr [ 0 ] { \sigma^2 \partial ^ 2 \phi } {   x ^ 2
\partial ^ 2   {  \phi } }+ FSGT
      \end{align*}

    \item \begin{align*}
          &2 \inpr [ 0  ] { \sigma x ^ 2 \partial \phi - \sigma \partial ^ 3
\phi } { \partial \sigma   x  ^ 2    { \phi }   - 3 \partial \sigma  \partial
^ 2    { \phi } }\\
& =  2 \inpr [ 0 ] { \sigma
 x ^ 2 \partial  \phi } { \partial \sigma . x ^ 2 .   { \phi }   } - 6 \inpr
[ 0 ] { \sigma x ^ 2 \partial \phi } { \partial \sigma . \partial ^ 2   {
\phi }  } - 2 \inpr[ 0 ] { \sigma \partial ^ 3 \phi } { \partial \sigma . x ^ 2
.   { \phi }   }\\
          & \quad\quad\quad\quad\quad\quad\quad\quad\quad\quad\quad\quad\quad\quad\quad\quad\quad+ 6 \inpr [ 0 ] { \sigma \partial ^ 3 \phi } { \partial \sigma .
\partial ^ 2    { \phi }  } \\
          &=  - \inpr [ 0 ] { \sigma \partial^2 \sigma \phi} { x ^ 4   { \phi
} }  - 3 \inpr [ 0 ] { \sigma \partial ^ 2 \sigma \partial ^ 2 \phi } {
\partial ^ 2   { \phi } }+ FSGT
      \end{align*}
and
\item \begin{align*}
          - 2 \inpr [ 0 ] { \sigma x ^ 2 \partial \phi - \sigma \partial ^ 3
\phi } {   3 \partial ^ 2  \sigma  \partial   { \phi } } = & - 6 \inpr [ 0 ]
{ \sigma x ^ 2 \partial \phi } { \partial ^ 2 \sigma  \partial   { \phi } } +
6 \inpr [ 0 ] {  \sigma \partial ^ 3 \phi } { \partial ^ 2 \sigma \partial  
{  \phi } }\\
          = & -  6 \inpr [ 0 ] { \sigma \partial ^ 2 \sigma  \partial \phi } {
x ^ 2 \partial   { \phi }  } - 6 \inpr [ 0 ] { \sigma \partial ^ 2 \sigma
\partial ^ 2 \phi } { \partial^2   { \phi } }+ FSGT.
      \end{align*}
\end{enumerate}
Consequently,
      \begin{align*}
           \inpr[ 0 ] {  \Hop A ^ * \phi } {  \Hop A ^ *   { \phi } } = & \left \|
\sigma  x ^ 2 \partial \phi \right \| ^ 2 _ 0 + \left \| \sigma \partial ^ 3
\phi \right \| ^ 2 _ 0 +  2 \inpr [ 0 ] { \sigma ^ 2 \partial ^ 2 \phi } { x ^ 2
\partial ^ 2   { \phi } } - \inpr [ 0 ] { \sigma \partial ^ 2 \sigma \phi } {
x ^ 4   { \phi } } \\
&\quad\quad\quad\quad- 8 \inpr [ 0 ] { \sigma \partial ^ 2 \sigma \partial \phi
 } { x ^ 2 \partial   { \phi } } 
            - 9 \inpr [ 0 ] { \sigma \partial ^ 2 \sigma \partial ^ 2 \phi } {
\partial ^ 2   { \phi }  } + FSGT
      \end{align*}
Again,
      \begin{align*}
          2  \Hop L ^ * \phi = & \left ( x ^ 2 \sigma ^ 2 \partial ^ 2 \phi - \sigma
^ 2 \partial ^ 4 \phi \right ) + \left ( 4 \sigma x ^ 2 \partial \sigma \partial
\phi - 2 x ^ 2 b \partial \phi - 8 \sigma \partial \sigma \partial ^ 3 \phi + 2
b \partial ^ 3 \phi \right ) \\
          &+ \left ( 2  \sigma \partial ^ 2 \sigma x ^ 2 \phi - 12 \sigma
\partial \sigma \partial ^ 2 \phi \right ) +\left ( \text{Terms of order} \leq 2
\right ).
      \end{align*}
Recalling that $\Hop \phi =    x ^ 2 \phi - \partial ^ 2 \phi$, we have
      \begin{align*}
          &2 \inpr [ 0 ] {   \Hop \phi } {  \Hop L ^ *   { \phi } }\\
          &\quad\quad\quad\quad=   \inpr [ 0 ] {
x ^ 2 \phi - \partial ^ 2 \phi } { x ^ 2 \sigma ^ 2 \partial ^ 2   { \phi } -
\sigma ^ 2 \partial ^ 4   { \phi } } + \inpr [ 0 ] { x ^ 2 \phi - \partial ^
2 \phi } { 2 \sigma \partial ^ 2 \sigma x ^ 2   { \phi }  - 12 \sigma
\partial \sigma \partial ^ 2   { \phi } } \\
          &\quad\quad\quad\quad+ \inpr [ 0 ] { x ^ 2 \phi - \partial ^ 2 \phi } { 4 \sigma  x ^ 2
\partial \sigma \partial   { \phi } - 2 x ^ 2 b  \partial   { \phi } - 8
\sigma \partial \sigma \partial ^ 3   { \phi } + 2 b \partial ^ 3   { \phi
} } + FSGT.
      \end{align*}
We focus on the first three terms. Note that
\begin{enumerate}[label=(\roman*)]
    \item \begin{align*}
          &\inpr [ 0 ] { x  ^ 2 \phi - \partial ^ 2 \phi } { x ^ 2 \sigma ^  2
\partial ^ 2   { \phi } -  \sigma ^ 2 \partial ^ 4   { \phi } }\\
&\quad\quad\quad\quad\quad\quad\quad=   -
\left ( \left \| \sigma  x ^ 2  \partial
 \phi \right \| ^ 2 _ 0 + \left \| \sigma \partial ^ 3 \phi \right\| ^ 2 _ 0 +
2 \inpr [
 0 ] { \sigma ^  2 \partial  ^ 2  \phi} {   x ^ 2  \partial ^ 2   { \phi } }
\right ) + \inpr [ 0 ] { \sigma \partial ^ 2 \sigma \phi } { x ^  4   { \phi
}  } \\
          & \quad\quad\quad\quad\quad\quad\quad\quad\quad+ 4 \inpr [ 0  ] { \sigma \partial ^ 2 \sigma \partial \phi } { x ^
2 \partial   { \phi } } + \inpr [ 0 ] { \sigma \partial ^ 2 \sigma \partial ^
2 \phi } { \partial ^ 2  { \phi } }+ FSGT
      \end{align*}
    \item \begin{align*}
          &\inpr [ 0 ] {  x ^ 2 \phi - \partial ^ 2 \phi } { 2 \sigma \partial ^
2 \sigma x ^ 2   { \phi }  - 12 \sigma \partial \sigma \partial ^ 2   {
\phi } }\\
&=  2 \inpr [ 0 ] { \sigma \partial ^ 2 \sigma \phi } { x ^ 4   {
\phi } } + 12 \inpr [ 0  ] { \sigma \partial ^ 2 \sigma \partial ^ 2 \phi } {
\partial ^ 2   { \phi } } + 14 \inpr [ 0 ] { \sigma \partial ^ 2 \sigma
\partial \phi } { x ^ 2 \partial   { \phi } }
     + FSGT \end{align*}
    and
    \item \begin{align*}
         &\inpr [ 0 ] {  x ^ 2 \phi - \partial ^ 2 \phi } { 4 \sigma x ^ 2
\partial \sigma \partial   { \phi }  - 2  x  ^ 2 b \partial   { \phi }  -
8 \sigma \partial \sigma \partial ^ 3   { \phi }  + 2 b \partial ^ 3   {
\phi }  }\\
&\quad\quad\quad\quad\quad\quad\quad\quad\quad\quad\quad=  - 2 \inpr [ 0 ] { \sigma \partial ^ 2 \sigma \phi } { x ^ 4   {
\phi } } - 10 \inpr [ 0 ] { \sigma \partial ^ 2 \sigma \partial \phi } { x ^ 2
\partial   { \phi } } \\
         & \quad\quad\quad\quad\quad\quad\quad\quad\quad\quad\quad\quad\quad\quad\quad\quad- 4  \inpr [ 0 ] { \sigma \partial ^ 2 \sigma \partial ^ 2 \phi } {
\partial ^ 2   { \phi } }+ FSGT.
      \end{align*}
\end{enumerate}
    Consequently,
      \begin{align*}
           2  \inpr [ 0  ] {   \Hop \phi } {  \Hop L ^  *   { \phi }  } = & - \left (
 \left \| \sigma x ^ 2 \partial \phi \right\| ^ 2 _ 0 + \left \| \sigma \partial
^ 3 \phi \right \| ^ 2 _ 0 +  2 \inpr [ 0 ] { \sigma ^ 2 \partial ^ 2 \phi }{ x
^ 2 \partial ^ 2   { \phi }  } \right )  + \inpr [ 0  ] {  \sigma \partial ^
2  \sigma. \phi } { x ^ 4   { \phi } }\\
&+ 8 \inpr [ 0 ] { \sigma \partial ^ 2
\sigma \partial \phi  } { x ^ 2 \partial   { \phi }  }  + 9 \inpr [ 0  ]{  \sigma \partial ^ 2 \sigma \partial ^ 2 \phi } {
\partial ^ 2   { \phi } } + FSGT,
      \end{align*}
and hence, $2 \inpr [ 0 ] {  \Hop \phi } {  \Hop L ^ \ast   { \phi } } + \inpr [
0 ] {  \Hop A ^ \ast \phi } {   \Hop A ^ \ast   { \phi} }$ is an FSGT. The result follows.
\end{proof}

\begin{proof}[Proof of Lemma \ref{H^P-lemma}]
We apply the principle of mathematical induction. In Lemma \ref{Mono-ineq-gen-functin-P012}, we have the result for $p = 2$. We treat this as the base case for our induction. 

Now, suppose result is true for $ p = k $,  i.e.
    \begin{align*}
          \Hop ^  k   \phi =& \sum _ { j  = 0  } ^  { k }    (  - 1 )  ^ {  k -  j }
\binom{ k } { j  }   x  ^  { 2  j }  \partial ^ { 2 ( k - j  ) }  \phi\\
&-  2 \sum _ { r  = 1 } ^ { k -  1 }   \left ( \sum_ { j = 1 } ^ { r }   (- 1 )
^ { r -
 j }  \binom {  r } {  j}   .   2 j   .   x ^ { 2  j - 1 }  \partial ^ { 2 ( r
-j ) + 1 }    \Hop ^ { k - 1 - r }  \phi  \right ) + (\text{Terms of order} \leq 2 k - 4 )\\
    \end{align*}
    Now, for $ p  =   k + 1 $   we have,
    \begin{align*}
         \Hop ^ { k + 1 }  \phi = & \sum _ { j = 0 } ^ { k } ( - 1 ) ^ { k - j }
\binom { k }{ j }  \Hop ( x ^ { 2 j } \partial ^ { 2 ( k - j ) }  \phi )\\
&\quad\quad\quad\quad\quad-  2 \sum
_ { r = 1 } ^ { k - 1 } \left ( \sum _ { j = 1 } ^ { r }  ( - 1 ) ^ { r - j }
\binom { r } { j }   .  2j .  \Hop (  x ^ { 2 j - 1 } \partial ^ { 2  ( r - j ) + 1
}    \Hop ^ { k - 1 - r }  \phi )  \right ) \\
        = & \sum _ { j = 0 } ^ { k } (- 1 )  ^ { k - j } \binom { k }  { j }  \Hop (
x ^ { 2 j } \partial ^ {2 (k - j ) } \phi )\\
&\quad\quad\quad\quad\quad\quad- 2 \sum _ { r = 1 } ^ { k - 1 }
\left ( \sum _ { j = 1 } ^ { r } ( - 1 ) ^ {r - j } \binom { r } { j } . 2 j . x
^ { 2 j - 1 } \partial ^ { 2 ( r - j ) + 1 }  \Hop ^ { k - r } \phi  \right ).
    \end{align*}
Since,
    \begin{align*}
        \sum _ { j =  0 } ^ { k } ( - 1 ) ^ { k - j }\binom { k } { j }&  \Hop ( x ^
{ 2 j } \partial ^ { 2 ( k - j ) } \phi )\\
 =&   \sum _ { j = 0 } ^ { k } ( - 1 )
^ { k-j } \binom { k } { j }  \Hop \left ( x ^  { 2 j } \partial ^ { 2  ( k - j ) }
\phi \right ) \\
        = & \sum _ { j = 0 } ^ { k } ( - 1 ) ^ { k - j } \binom { k } { j }
\left ( x ^ { 2 j + 2 } \partial ^ { 2 ( k - j ) } \phi - \partial ^ 2 ( x ^ { 2
j } \partial ^ { 2 ( k - j ) } \phi ) \right ) \\
        =& \sum _ { j = 0 } ^ { k } ( - 1 ) ^ { k - j } \binom { k } { j } \left
( x ^ { 2 j + 2 } \partial ^ { 2 ( k - j ) } \phi - x ^ { 2 j } \partial ^ { 2 (
k + 1 - j ) } \phi - 2 . 2 j . x ^ { 2 j - 1 } \partial ^ { 2 ( k - j ) + 1 }
\phi \right ) \\
         = &   \sum _ { j = 0 } ^ { k } ( - 1 ) ^ { k - j } \binom { k }{ j } x
^ { 2 j + 2 } \partial ^ { 2 ( k - j ) } \phi - \sum _ { j = 0 } ^ { k } ( - 1 )
^ { k - j } \binom{ k } { j } x ^ { 2 j } \partial ^ { 2 ( k + 1 - j ) } \phi \\
          &- 2 \sum _ { j = 0 } ^ { k } ( - 1 ) ^ { k - j } \binom { k } { j } 2
j . x ^ { 2 j - 1 } \partial ^ { 2 ( k - j ) + 1 } \phi\\
         =& \sum _ { j = 1 } ^ { k + 1 } ( - 1 ) ^ { k - j + 1 } \binom { k } {
j - 1 } x ^ { 2 j } \partial ^ { 2 ( k - j + 1 ) } \phi + \sum _ { j = 0 } ^ { k
} ( - 1 ) ^ { k + 1 - j } \binom { k } { j } x ^ { 2 j } \partial ^ { 2 ( k + 1
- j ) } \phi\\
         & - 2 \sum _ { j = 0 } ^ { k } ( - 1 ) ^ { k - j } \binom { k } { j } 2
j . x ^ { 2 j - 1 } \partial ^ { 2 ( k - j ) + 1 } \phi \\
         = &   ( - 1 ) ^ { k + 1 } \partial ^ { 2 ( k + 1 ) } \phi + \sum _ { j
= 1 } ^ { k } ( - 1 ) ^ { k + 1 - j } \left ( \binom { k } { j - 1 } + \binom {
k } { j } \right ) x ^ { 2 j } \partial ^ { 2 ( k + 1 - j ) } \phi \\
         & + x ^ { 2 j } \phi - 2 \sum _ { j = 0 } ^ { k } ( - 1 ) ^ { k - j }
\binom { k } { j   } 2 j . x ^ { 2 j - 1 }  \partial ^ { 2 ( k - j ) + 1 } \phi
\\
         = &  \sum _ { j = 0 } ^ { k + 1 } ( - 1 ) ^ { k + 1 - j } \binom{ k + 1
} { j } x ^ { 2 j } \partial ^ { 2 ( k + 1 - j ) } \phi - 2 \sum _ { j = 0 } ^ {
k } ( - 1 ) ^ { k - j } \binom { k } { j } 2 j . x ^ { 2 j - 1 } \partial ^ { 2
( k -   j ) + 1 } \phi,
    \end{align*}
and hence,
    \begin{align*}
          \Hop ^ { k + 1 } \phi = & \sum _ { j = 0 } ^ { k + 1 } ( - 1 ) ^ { k + 1 -
j } \binom { k + 1 } { j } x ^ { 2 j } \partial ^ { 2 ( k + 1 - j ) } \phi\\
&- 2
\sum _ { r = 1 } ^ { k } \left ( \sum _ { j = 1 } ^ { r } ( - 1 ) ^ { r - j }
\binom{ r } { j } . 2 j . x ^ { 2 j - 1 } \partial ^ { 2 ( r - j ) + 1 }  \Hop ^ { k
- r } \phi  \right ) \\
        & + (\text{Terms of order} \leq  2 ( k + 1 ) - 4 ).
    \end{align*}
Applying the principle of mathematical induction, the proof is completed.
\end{proof}

\begin{proof}[Proof of Theorem \ref{Mono-ineq-gen-functin}]
We prove the result using Lemma \ref{H^P-lemma} and the principle of mathematical induction. First, note that the inequality holds for $p = 0$, by an application of integration by parts.

Now, observe that
    \begin{align*}
         \Hop A^\ast \phi = & \left (  \sigma \partial ^ 3  \phi -  x ^ 2  \sigma
\partial \phi \right ) + \left ( 3 \partial \sigma \partial ^ 2 \phi - x ^ 2
\partial \sigma . \phi \right ) + 3 \partial ^ 2 \sigma \partial \phi + GT\\
          A^\ast  \Hop \phi = &  \left (   \sigma \partial ^ 3 \phi - x ^ 2  \sigma
\partial \phi \right ) + \left (  \partial \sigma \partial ^ 2 \phi - x ^ 2
\partial \sigma . \phi - 2 x \sigma \phi \right ) + GT,
    \end{align*}
    and hence
    \begin{align*}
        \left (    \Hop    A^\ast - A^\ast  \Hop  \right ) \phi = 2  \partial \sigma \partial ^ 2
\phi - 2 x \sigma \phi + 3 \partial ^ 2  \sigma \partial \phi +GT.
    \end{align*}
Similarly, we have
\begin{align*}
     \left (  \Hop L^\ast-L^\ast  \Hop \right )  \phi = & 2 \left(   \sigma ^ 2 x \partial  \phi -
\sigma \partial \sigma \partial ^ 3 \phi \right ) - 5 \sigma \partial ^ 2 \sigma
\partial ^ 2 \phi + GT.
\end{align*}
Now assume that the result is true for $p$. Then we can write
\begin{align*}
    2 \inpr [ p + 1 ] { \phi } { L^\ast \phi}   +  \left \| A ^\ast\phi \right \| ^ 2 _ {
p + 1 } =  &  \left ( 2 \inpr [ p ] {    \Hop   \phi } { L^\ast   \Hop \phi } + \left \| A^\ast  \Hop
\phi \right \| ^ 2 _ { p }  \right ) 
    +  2 \inpr [ p ] {  \Hop \phi } { \left (   \Hop L^\ast - L^\ast  \Hop \right ) \phi }\\
    &\quad\quad\quad\quad\quad+ \left\|(
 \Hop A^\ast - A^\ast  \Hop ) \phi \right \| ^ 2 _ { p } +  2 \inpr [ p ] { A^\ast  \Hop \phi } {  \left (
 \Hop A^\ast - A^\ast  \Hop \right )  \phi }.
\end{align*}
For $\phi \in \Sc(\R)$, we have $\Hop \phi \in \Sc(\R)$. The relevant bound for the first term on the right hand side then follows from the induction hypothesis. For the other terms, observe that
\begin{enumerate}[label=(\roman*)]
    \item $\left \| (     \Hop   A^\ast - A^\ast  \Hop   ) \phi \right \| ^ 2 _ { p } =   FSGT$

    \item \begin{align*}
         2 \inpr [ p ] {  \Hop  \phi } { \left (    \Hop    L^\ast - L^\ast    \Hop  \right ) \phi }=
&  4 \inpr [ p ] { x ^ 2  \phi } { \sigma ^ 2 x  \partial \phi } - 4 \inpr [ p ]
{ x ^ 2  \phi } { \sigma \partial \sigma \partial ^ 3  \phi } - 1 0  \inpr [ p ]
{ x ^ 2  \phi } { \sigma ^ 2  \partial ^ 2 \sigma\partial ^ 2  \phi }\\
         & - 4  \inpr [ p ] { \partial ^ 2  \phi } { \sigma ^ 2 x  \partial \phi
} + 4 \inpr [ p ] { \partial ^ 2   \phi} {  \sigma  \partial \sigma \partial ^ 3
\phi } + 10  \inpr[ p ] { \partial ^ 2 \phi } { \sigma \partial ^ 2 \sigma
\partial ^ 2  \phi }\\
&+ FSGT
\end{align*}
and 
\item     \begin{align*}
        2 \inpr [ p ] { A^\ast  \Hop \phi } {  \left (    \Hop   A ^\ast-  A^\ast  \Hop \right )  \phi } =
&  4 \inpr [ p ] { \sigma \partial ^ 3  \phi } { \partial \sigma \partial ^ 2
\phi } + 6 \inpr [ p  ] { \sigma \partial ^ 3  \phi } { \partial ^ 2 \sigma
\partial \phi } \\
         & - 4 \inpr [ p  ] { x ^ 2  \sigma \partial \phi } { \partial \sigma
\partial ^ 2 \phi } - 6  \inpr [ p ] { x ^ 2  \sigma \partial \phi } { \partial
^ 2  \sigma \partial \phi }+ FSGT.
    \end{align*}
\end{enumerate}
In the above computation, we have used the observation that
  $ \inpr [ p ] { x ^ 2   \phi } { \sigma ^ 2 x  \partial \phi },\;  \inpr [ p ]
{ \partial ^ 2  \phi } { \sigma ^ 2 x  \partial \phi } $ can be reduced as
a finite sum  of good terms. Now,
\begin{align*}
    2 \inpr [ p + 1 ] { \phi } { L^\ast \phi }  +  \left \| A^\ast  \phi \right \| ^ 2 _ {
p + 1 } =  &  \left ( 2   \inpr [ p ] {  \Hop  \phi } { L^\ast  \Hop \phi }  + \left \| A ^\ast  \Hop
\phi \right \| ^ 2 _ p \right )  + \left ( S _ 6 ' + S _ 5' \right )  + \left (
S _ 6'' + S _ 5 '' \right )
\end{align*}
   where
   \begin{align*}
        S  _ 6 ' = &   -   4   \inpr [ p ] { x ^ 2  \phi } { \sigma \partial
\sigma \partial  ^ 3  \phi } - 4\inpr [ p ] { x ^ 2  \sigma \partial \phi } {
\partial \sigma \partial ^ 2 \phi } \\
           S _ 5 ' = & -  10   \inpr [ p ] { x ^ 2  \phi } { \sigma \partial ^ 2
 \sigma \partial ^ 2 \phi } - 6 \inpr [ ] { x ^ 2   \sigma \partial \phi } {
\partial ^ 2  \sigma \partial \phi } \\
             S _ 6 '' = &  4 \inpr [ p ] { \partial ^ 2  \phi } { \sigma
\partial \sigma \partial ^ 3  \phi } + 4 \inpr [ p ] { \sigma\partial ^ 3  \phi
}{ \partial \sigma  \partial ^ 2  \phi } \\
           S _ 5 '' =&  10 \inpr [ p ] { \partial ^ 2  \phi } { \sigma \partial
^ 2  \sigma  \partial ^ 2  \phi } + 6  \inpr [ p  ] {  \sigma \partial ^ 3  \phi
}{\partial ^ 2  \sigma \partial \phi }
   \end{align*}
For $S  _ 6 '$, note that
   \begin{align*}
       \inpr [ p  ] { x ^ 2   \phi} {  \sigma \partial \sigma \partial ^  3
\phi } = & \inpr [ 0  ] {  \Hop ^ {  2 p } (  x ^ 2  \phi ) } {  \sigma \partial
\sigma \partial ^  3 \phi }\\
         = &    -  \sum  _ { i =  -  1 }  ^ {  2 p  - 1  } \binom { 2 p  } {  i
+
 1 }  \left(  2  p -  i - \frac { 5 }  { 2  } \right ) \inpr [  0 ] {
 x ^  { 2  i +  4 } \partial^  { 2  p -  i }  \phi}  { \sigma \partial ^  2
\sigma\partial ^ { 2 p  - i}   \phi } + FSGT.
   \end{align*}
and
  \begin{align*}
       \inpr [ p  ] {  x ^  2 \sigma \partial \phi } {  \partial \sigma \partial
^ 2  \phi } = &  \inpr [ 0  ] { \Hop^  { 2 p  }(  x  ^ 2  \sigma \partial \phi) }  {
\partial \sigma \partial ^  2 \phi  }\\
         = &   \sum_ {  i = -  1  } ^ {  2  p - 1 } \binom{  2 p  } { i +  1 }
\left ( 2 p - i - \frac {  3 } {  2 } \right ) \inpr [0  ] {  x  ^ { 2  i +  4 }
\partial ^  { 2 p  - i }  \phi } {  \sigma \partial ^ 2  \sigma\partial ^ {  2 p
 - i }   \phi }+ FSGT.
   \end{align*}

For $S  _ 5 '$, note that
\begin{align*}
       \inpr [ p  ] { x ^ 2  \phi } { \sigma \partial ^ 2  \sigma \partial ^  2
\phi }=
 & \inpr [ 0 ]  {  \Hop  ^ {  2 p  } (  x ^ 2  \phi )  } {  \sigma \partial^ 2
\sigma \partial ^  2   \phi }\\
         = &  - \sum _ { i =- 1 } ^ {  2p  - 1 } \binom  { 2 p } {  i + 1 }
\inpr [  0 ] {  x  ^ { 2i  + 4 }  \partial^ {  2 P  - i  }  \phi } {  \sigma
\partial ^  2 \sigma \partial ^  { 2 p -  i } \phi }+ FSGT
   \end{align*}
and
 \begin{align*}
       \inpr [ p  ] {  x ^ 2 \sigma  \partial \phi } {   \partial ^ 2   \sigma
\partial \phi }= &  \inpr[ 0  ] {  \Hop  ^  { 2 p } ( x  ^ 2  \sigma\partial \phi )  } {
 \partial ^ 2 \sigma \partial \phi  } \\
         = &  \sum _ { i = - 1 }  ^ { 2  p - 1 } \binom { 2 p  } { i + 1} \inpr
[ 0  ] {x
 ^ { 2  i + 4 } \sigma\partial ^ { 2 p -  i }  \phi } {  \partial ^ 2  \sigma
\partial ^
 {  2 p - i }  \phi }+FSGT.
   \end{align*}

For $S  _ 6 ''$, note that 
 \begin{align*}
       \inpr [ p  ] { \partial ^ 2  \phi } { \sigma \partial \sigma \partial ^ 3
 \phi } = & \inpr [ 0 ] {  \Hop ^ { 2 p } ( x ^ 2  \phi ) } { \sigma \partial \sigma
\partial ^ 3  \phi } \\
         =& \sum _ { i  = - 2  } ^ {  2 p  - 2 } \binom { 2  p  } { i + 2  }
\left ( 2  p -  i - \frac {  5 }  { 2  } \right ) \inpr [ 0  ] {   x ^  { 2  i
 + 4} \partial ^  { 2 p  - i  }  \phi} { \sigma \partial ^  2  \sigma \partial ^
{ 2  p -
  i }   \phi }+ FSGT
   \end{align*}
   and
   \begin{align*}
       \inpr [ p  ] {  \sigma \partial ^ 3  \phi } {  \partial \sigma \partial ^
2  \phi } = & \inpr [ 0 ]  {  \Hop  ^ {  2 p  } (  \sigma \partial ^ 3  \phi ) } {
\partial \sigma \partial ^ 2 \phi  }\\
         = &  - \sum _ { i  = -2 }  ^ { 2 p  - 2 }  \binom{ 2 p }  { i  + 2 }
\left ( 2  p- i - \frac {  3 } {  2 } \right ) \inpr [ 0 ]  { x ^  { 2 i +  4 }
\partial
 ^ { 2  p  - i }   \phi } { \sigma \partial ^  2 \sigma \partial^  { 2   p -  i
}  \phi }+ FSGT.
   \end{align*}
   
For $S  _ 5 ''$, note that
\begin{align*}
       \inpr [  p ] { \partial ^  2  \phi }  {  \sigma \partial ^ 2  \sigma
\partial ^ 2  \phi } = &  \inpr[ 0  ] {  \Hop  ^ { 2  p } ( \partial ^ 2  \phi ) }
{  \sigma \partial ^ 2 \sigma \partial ^  2 \phi }\\
         = &  \sum _ {  i = - 2 }  ^  { 2 p - 2 }  \binom{ 2 p }  {  i +  2 }
\inpr [ 0  ] {  x ^ { 2  i + 4  } \partial ^ { 2  p -  i }  \phi }  { \sigma
\partial ^ 2 \sigma \partial ^ { 2  p - i }  \phi }+ FSGT
   \end{align*}

and

   \begin{align*}
       \inpr [ p ] { \sigma \partial ^ 3  \phi} {   \partial ^ 2  \sigma
\partial \phi } = &  \inpr [  0 ] {  \Hop ^ { 2 p } ( \sigma \partial ^ 3 \phi ) } {
 \partial ^ 2 \sigma \partial \phi } \\
         = &   -\sum _ { i  = - 2 } ^ { 2 p - 2 } \binom { 2 p } { i + 2 } \inpr
[ 0 ] { x ^ { 2 i + 4 } \partial ^ { 2 p - i }   \phi } { \sigma \partial ^  2
\sigma \partial ^ { 2 p - i } \phi }+ FSGT.
   \end{align*}
Consolidating the computations, we have
\begin{align*}
    S _ 5 '= & 4 \sum _ { i = - 1 } ^ { 2 p - 1 } \binom { 2 p } { i + 1 } \inpr
[ 0 ] { x ^ { 2 i + 4 } \partial ^  { 2 p - i }  \phi } { \sigma \partial ^ 2
\sigma \partial ^ { 2 p - i }  \phi }\\
    S _ 6 ' = & 4 \sum _ { i = - 1 } ^ { 2 p - 1 } \binom {2 p } { i + 1 }
\left( 2 p - i - \frac { 5 } { 2 } \right ) \inpr [ 0 ] { x ^ { 2 i + 4 }
\partial ^ { 2 p  -i }  \phi } { \sigma \partial ^ 2 \sigma\partial ^ { 2 p - i
} \phi }\\
    & - 4 \sum _ { i  =-  1 } ^ { 2  p - 1 } \binom{ 2 p }  { i +  1 } \left ( 2
p  - i - \frac { 3 } { 2 } \right ) \inpr [ 0 ] { x ^ { 2 i + 4 } \partial ^ { 2
p  - i } \phi}{\sigma \partial ^ 2  \sigma\partial ^ { 2 p - i } \phi }\\
    = & - 4 \sum _ { i = - 1 } ^ { 2 p - 1 }\binom{ 2 p } { i + 1 } \inpr [ 0 ]
{ x ^ { 2 i + 4 } \partial ^ { 2 p - i }  \phi} { \sigma \partial ^ 2 \sigma
\partial ^ { 2 p - i } \phi }\\
    S _ 5 '' = & 4 \sum _ { i = - 2 } ^ { 2 p - 2 } \binom { 2 p } { i + 2 }
\inpr [ 0 ] { x ^ { 2 i + 4 } \partial ^ { 2 p - i }  \phi } { \sigma \partial ^
2   \sigma \partial  ^ { 2 p - i }  \phi }\\
    S _ 6 '' = &  - 4 \sum _ { i = - 2 } ^ { 2 p - 2 } \binom { 2 p } { i + 2 }
\inpr [ 0 ] { x ^ { 2 i + 4 } \partial ^ { 2 p - i  }  \phi } { \sigma \partial
^ 2 \sigma \partial ^ { 2 p - i }  \phi }.
\end{align*}
Therefore we have
\begin{align*}
    2 \inpr [ p + 1 ] { \phi } { L^\ast \phi }  + \left \| A ^\ast\phi \right \|  ^ 2 _ {
p + 1 } = & \left ( 2  \inpr [ p ] {   \Hop  \phi } { L^\ast  \Hop \phi } + \left \| A^\ast  \Hop \phi
\right \| ^ 2 _ { p }  \right) + FSGT.
\end{align*}
For $\phi \in \Sc(\R)$, we have $\Hop \phi \in \Sc(\R)$. The relevant bound for the first term on the right hand side then follows from the induction hypothesis. The proof is now complete.
\end{proof}

\subsection{Third Approach}\label{third approach}
\begin{proof}[Proof of Theorem \ref{proof-by-3-line lemma in affine case}]
Let $\phi \in \Sc(\R) \subset \Sc(\R; \C)$. To simplify the illustration of the technique, we take $\beta = 1$ and $\alpha = \gamma = \delta = 0$, which yields $L^\ast \phi =\frac{1}{2} \partial ^ 2 \phi$ and $ A ^ \ast \phi =
- \partial \phi $.

     Consider the complex valued function 
     \[ F ( z ) = 2 \inpr [ 0,\C ] { \phi }{  \Hop
^ { \Bar { z } } L ^ \ast   \Hop ^ { - \Bar { z } }  { \phi } }   + \inpr [ 0,\C ]
 {  \Hop ^ { z } A ^ \ast   \Hop ^ { - z } \phi } {  \Hop ^ { \Bar{ z }  } A ^ \ast   \Hop ^ { -
\Bar{ z } }  { \phi } }, z \in S.\]
By expanding the right hand side of the above definition, we have the analiticity of the function $F$ on $0 <  Re ( z ) < 1$ and continuity on $ 0
\leq Re ( z ) \leq   1 $. Now, for any $y \in \R$,
\begin{align*}
    F ( i y ) = & 2 \inpr [ 0,\C ] { \phi } {  \Hop ^ { - i  y } L ^ \ast  \Hop ^  { i y } { \phi } } +\inpr [ 0,\C ] {  \Hop ^ { i y } A ^ \ast  \Hop ^ { - i y } \phi } {  \Hop^ {
- i y } A ^ \ast   \Hop ^ {  i y } { \phi } } \\
    = & 2 \inpr [ 0,\C ] {  \Hop ^ { i y } \phi } { L ^ \ast   \Hop ^ { i y }  { \phi }
} + \inpr [ 0,\C ] {  \Hop ^ { 2 i y } A ^ \ast  \Hop ^ { - i y } \phi } { A ^ \ast  \Hop ^
 { i y }  { \phi } } \\
     = & 2 \inpr [ 0,\C ] {  \Hop ^ { i y } \phi } { L ^ \ast   \Hop ^ { i y }  { \phi
} } + \inpr [ 0,\C ] { \left (  \Hop ^ { 2 i y } A ^ \ast  \Hop ^ { - 2 i y} \right )  \Hop ^ {
i y } \phi } { A ^ \ast   \Hop ^ { i y }  { \phi } } \\
     = &  \inpr [ 0,\C ] {  \Hop ^ { i y } \phi } { \partial ^ 2   \Hop ^ { i y }    {
\phi } } +\inpr [ 0,\C ] { \left (  \Hop ^ { 2 i y } A ^ \ast  \Hop ^ { - 2 i y } \right )  \Hop
^ { i y } \phi} { A ^ \ast   \Hop ^ {  i y }   { \phi } } \\
      = & -  \inpr [ 0,\C ] { \partial  \Hop ^ { i y } \phi } { \partial   \Hop ^ { i y }
   { \phi } } + \inpr [ 0,\C ] { \left (  \Hop ^  { 2 i y } A ^ \ast  \Hop ^ { - 2 i y }
\right )  \Hop ^ { i y } \phi } { A ^ \ast   \Hop ^ { i y }    { \phi } } \\
       =&\inpr [ 0,\C ] { \left (  \Hop ^ {  2 i y  } A ^ \ast  \Hop ^ { - 2  i y } - A ^
\ast \right )  \Hop ^ {  i y } \phi } { A ^ \ast   \Hop ^ { i  y }    { \phi } }.
\end{align*}
Using Lemma \ref{bndblm1} and  and Proposition \ref{Hop-C-properties}, we have
\[\left | F  ( i y ) \right | \leq   C \|  \Hop ^ { i y } \phi \| _ {0,\C} \|  \Hop ^ { i
y }    { \phi } \| _ {0,\C} \\
    \leq   C \| \phi \| _{0,\C} ^ 2,\]
for some $C > 0$, independent of the choice of $\phi$. Again for $ z = 1 + i y, y \in \R $ , we have
\begin{align*}
     F ( 1 + i  y ) = & 2 \inpr [ 0,\C ] { \phi } {  \Hop ^ { 1 - i  y } L ^ \ast
  \Hop ^ { - 1 + i y }    { \phi } } + \inpr [ 0,\C ] {  \Hop ^ { 1 + i  y } A ^ \ast  \Hop ^
{  - 1 - i  y} \phi } {  \Hop  ^ {1   - i y  } A ^ \ast   \Hop ^  { -  1  + i y  }   
{ \phi } }\\
    = &  \inpr  [ 0,\C ] {   \Hop ^ { 1 + i y } \phi } {  \partial ^ 2    \Hop ^
 { - 1  + i  y }    { \phi } }  +  \inpr [ 0,\C ] { \left (  \Hop ^  { 2 (  1 + i y
) } A ^
 \ast  \Hop ^ { - 2  ( 1 + i  y ) } \right )  \Hop ^ {  1 +  i y } \phi } { A ^ \ast  \Hop
^
 { - 1 +  i y }    { \phi } } \\
      = & - \inpr [ 0,\C ]{ \partial  \Hop ^ { 1 + i  y } \phi } { \partial  \Hop ^ {
 -  1 + i  y }    { \phi } } +  \inpr [ 0,\C ] { \left (  \Hop ^ {  2 (  1 +  i y  )
} A ^ \ast  \Hop ^ { -  2 (  1 +i y  ) } \right )  \Hop ^  { 1  + i  y } \phi } {  A ^
\ast  \Hop ^ {
 - 1  + i  y  }    { \phi } } \\
       = & \inpr [ 0,\C ] { \left (  \Hop ^  { 2  (  1 + i  y ) }  A ^ \ast  \Hop ^ {  - 2
(  1 +
 i y  ) } - A ^ \ast \right )  \Hop  ^ {  1 + i  y } \phi } {  A ^ \ast  \Hop ^ {
 - 1  + i  y }    { \phi } }.
\end{align*}
and consequently,
\begin{align*}
    | F ( 1 +  i y ) | \leq &  C \| \phi \| _ {0,\C} \|    { \phi } \| _ {0,\C} \\
      = &  C \| \phi \| ^ 2 _ {0,\C},
\end{align*}
for some $C > 0$, independent of the choice of $\phi$. Therefore, by the Three Lines Lemma, we have
\[\left | F ( x  + i y ) \right | \leq  ( C \| \phi \| ^ 2  _ {0,\C} ) ^ { 1 - x
} ( C \| \phi \| ^ 2 _ {0,\C} ) ^ x 
    \leq   C \| \phi \| ^ 2 _ {0,\C},\]
for all $ z \in  S = \left \{ x + i y = z \in \C : \; 0 \leq x \leq 1 \right \}$. In particular, for $ 0 <  p < 1 $, we  have  \[ \left | F  ( p  + i 0  )
\right | \leq     C \| \phi \| ^ 2  _ {0,\C} \]
and hence
\[2 \inpr [ 0,\C ] { \phi } {  \Hop ^  p L ^ \ast  \Hop ^ {  - p } \phi } +  \inpr[ 0,\C ] {
 \Hop ^ { p } A ^ \ast  \Hop ^ { - p } \phi } {  \Hop ^ p A ^ \ast  \Hop ^ {  - p } \phi } \leq
   C \| \phi \| ^ 2 _ {0,\C},\]
   with $C > 0$, independent of the choice of $\phi$. As $\phi$ is real-valued, using \eqref{real-complex-inner-pr-consistent}, we have
   \[2 \inpr [ 0] { \phi } {  \Hop ^  p L ^ \ast  \Hop ^ {  - p } \phi } +  \inpr[ 0 ] {
 \Hop ^ { p } A ^ \ast  \Hop ^ { - p } \phi } {  \Hop ^ p A ^ \ast  \Hop ^ {  - p } \phi } \leq
   C \| \phi \| ^ 2 _ {0}.\]
Replacing $\phi$ by $\Hop^p \phi$, we get the result for $0 < p < 1$. The argument is similar for other non-negative non-integer values of $p$.
\end{proof}

\begin{remark}
    At the present stage of this approach, we are able to handle affine multipliers $\sigma$ and $b$; but we hope to extend the methodology to incorporate more general multipliers  in a future work. Theorem \ref{proof-by-3-line lemma in affine case} yields an alternative proof of Theorem \ref{Mono-ineq-affine-Form}, for non-negative non-integer $p$ regularities.
\end{remark}

\textbf{Funding:} A. K. Nath would like to acknowledge the fact that he was supported by the University Grants Commission (Government of India) Ph.D research Fellowship. S. Bhar would like to acknowledge the fact that he was partially supported by the Matrics grant MTR/2021/000517 from the Science and Engineering Research Board (Department of Science \& Technology, Government of India). 

\bibliographystyle{plain}
\bibliography{ref}

\end{document}